\documentclass[12pt]{amsart}
\usepackage[margin=1in]{geometry}
\usepackage{graphicx, amssymb, amsmath,amsfonts,mathrsfs, hyperref,color,tikz-cd}
\usepackage{verbatim}
\title{On the elliptic Hall category}
\author{Roman Bezrukavnikov, Ivan Losev, Xinchun Ma}

\newcommand{\Hcal}{\mathcal{H}}
\newcommand{\AH}{\mathcal{AH}}
\newcommand{\CS}{\mathcal{CS}}
\newcommand{\id}{\mathrm{id}}
\newcommand{\Lie}{\mathrm{Lie}}
\newcommand{\qc}{\mathrm{qc}}
\newcommand{\pt}{\mathrm{pt}}
\newcommand{\Nilp}{\mathcal{N}}
\newcommand{\C}{\mathbb{C}}
\newcommand{\Cfrak}{\mathfrak{C}}
\newcommand{\Unip}{\mathcal{U}}
\usepackage[color,matrix,arrow]{xy}
\newcommand{\ol}{\overline}

\newcommand{\Z}{\mathbb{Z}}
\newcommand{\Q}{\mathbb{Q}}

\newcommand{\Pbb}{\mathbb{P}}
\newcommand{\Kbb}{\mathbb{K}}
\newcommand{\Gbb}{\mathbb{G}}
\newcommand{\Ffrak}{\mathfrak{F}}
\newcommand{\Fcal}{\mathcal{F}}
\newcommand{\Vcal}{\mathcal{V}}

\newcommand{\Gcal}{\mathcal{G}}
\newcommand{\Ocal}{\mathcal{O}}
\newcommand{\Bcal}{\mathcal{B}}
\newcommand{\Dcal}{\mathcal{D}}

\newcommand{\Hrm}{\mathrm{H}}

\newcommand{\Hom}{\mathrm{Hom}}

\newcommand{\im}{\mathrm{Im}}

\newcommand{\gr}{\mathrm{gr}}

\newcommand{\Desc}{\mathrm{Desc}}
\newcommand{\Ind}{\mathrm{Ind}}
\newcommand{\bfrak}{\mathfrak{b}}
\newcommand{\nfrak}{\mathfrak{n}}

\newcommand{\gfrak}{\mathfrak{g}}
\newcommand{\g}{\mathfrak{g}}

\newcommand{\tfrak}{\mathfrak{t}}

\newcommand{\pfrak}{\mathfrak{p}}
\newcommand{\qfrak}{\mathfrak{q}}

\newcommand{\Xfrak}{\mathfrak{X}}
\newcommand{\Yfrak}{\mathfrak{Y}}

\newcommand{\Zcal}{\mathcal{Z}}

\newcommand{\Lbb}{\mathbb{L}}
\newcommand{\Nbb}{\mathbb{N}}
\newcommand{\olNbb}{\ol{\mathbb{N}}}
\newcommand{\Mbb}{\mathbb{M}}

\newcommand{\tr}{\mathrm{tr}}
\newcommand{\Ad}{\mathrm{Ad}}

\newcommand{\Kcal}{\mathcal{K}}

\newcommand{\Dbb}{\mathbb{D}}
\newcommand{\Sym}{\mathrm{Sym}}

\newcommand{\Ccal}{\mathcal{C}}

\newcommand{\Fbb}{\mathbb{F}}
\newcommand{\Ebb}{\mathbb{E}}

\newcommand{\Ncal}{\mathcal{N}}
\newcommand{\Mcal}{\mathcal{M}}

\newcommand{\Stab}{\mathrm{Stab}}

\newcommand{\End}{\mathrm{End}}
\newcommand{\GL}{\mathrm{GL}}
\newcommand{\SL}{\mathrm{SL}}
\newcommand{\Spec}{\rm{Spec} }

\newcommand{\Lcal}{\mathcal{L}}
\newcommand{\Lbf}{\mathbf{L}}

\newcommand{\gl}{\mathfrak{gl}}
\newcommand{\sln}{\mathfrak{sl}_n}
\newcommand{\zfrak}{\mathfrak{z}}

\newcommand{\wt}{\widetilde}

\newcommand{\ind}{\mathrm{ind}}

\newcommand{\Ker}{\mathrm{Ker}}

\newcommand{\sign}{\mathrm{sign}}

\newcommand{\Coh}{\mathrm{Coh}}

\newcommand{\CH}{\mathrm{CH}}
\newcommand{\HC}{\mathrm{HC}}
\newcommand{\MHM}{\mathrm{MHM}}

\newcommand{\Rep}{\mathrm{Rep}}
\usepackage[dvipsnames,svgnames]{xcolor}
\theoremstyle{plain}
\newtheorem{thm}{Theorem}[section]
\newtheorem{prop}[thm]{Proposition}
\newtheorem{coro}[thm]{Corollary}
\newtheorem{lem}[thm]{Lemma}

\theoremstyle{definition}
\newtheorem{defn}[thm]{Definition}
\theoremstyle{remark}
\newtheorem{rem}[thm]{Remark}

\numberwithin{equation}{section} \unitlength=1mm 

\begin{document}
\maketitle
\begin{abstract}
In this paper we  study (the positive half of) the elliptic Hall category, a categorical version of the elliptic Hall algebra. Our main result is a computation of categorical versions of the generators in the coherent realization of the elliptic Hall category, i.e., via coherent sheaves on commuting stacks for $\operatorname{GL}_n\times \mathfrak{gl}_n$: we realize them as Hodge associated graded sheaves of certain twisted equivariant $D$-modules on the groups $\operatorname{GL}_n$.
\end{abstract}

{\centerline{\em To Victor Ginzburg with gratitude and admiration}}

\tableofcontents

\section{Introduction}
This paper deals with an {\it elliptic Hall category}, a categorical version of the elliptic Hall algebra. The latter has been studied quite extensively in the last 30 years or so. One exciting feature of this algebra is that it admits several different realizations:
\begin{itemize}
\item[(i)] As the {\it Ding-Iohara algebra}, a quantum group for the double Heisenberg algebra associated to $\mathfrak{gl}_1$, \cite{Ding_Iohara}.
\item[(ii)] As a suitable limit as $n\rightarrow \infty$ of spherical double affine Hecke algebras for $\operatorname{GL}_n$, \cite{SchiffmannVasserot1}.
\item[(iii)] As a suitable version of the Hall algebra of an elliptic curve over a finite field, \cite{BurbanSchiffmann}.
\item[(iv)] As a sort of a double of  the direct sum (over $n\geq 0$) of equivariant K-theories of commuting stacks for the Lie algebras $\mathfrak{gl}_n$, \cite{SchiffmannVasserot2}.
\end{itemize}
Let us review a few features of this algebra. First, it is an algebra over $\mathbb{C}(q,t)$. Second, it has generators labeled by $\mathbb{Z}^2\setminus \{(0,0)\}$. From here we get a triangular decomposition, where the positive part corresponds to the generators labeled by $(r,a)\in \mathbb{Z}^2$ with $r>0$. For example, in realization (iv), the positive part is the direct sum of the equivariant K-theories of commuting stacks for $\mathfrak{gl}_n$  with $r=n$. A natural question therefore is to categorify the elliptic Hall algebra or its positive part  (one could hope to recover a categorification of the former from that of the latter, compare with Kac-Moody categorifications).

We note that realizations (iii) and (iv) above suggest categorifications of the positive part: (iii) in terms of constructible sheaves on moduli spaces of vector bundles on the elliptic curves, while (iv) in terms of coherent sheaves on commuting stacks on $\mathfrak{gl}_n$'s. One expects that the resulting categorifications are equivalent, this equivalence can be thought of as a double affine relative of an equivalence of two categorifications of the affine Hecke algebra by the first named author, \cite{Bezrukavnikov_Hecke}.

In fact, there is the third candidate for being a categorification of the positive part of the elliptic Hall algebra coming from the trace categories of affine Hecke categories for $\operatorname{GL}_n$ that we now describe.

Assume for now that $G$ is an arbitrary connected reductive group and let $G/B$ be its flag variety. We write $\mathcal{Z}  $ for the Steinberg dg-scheme, i.e., the derived fiber product $T^*(G/B)\times^R_{\mathfrak{g}^*}T^*(G/B)$, where the map $T^*(G/B)\rightarrow \g^*$ is the moment map for the action of $G$. We consider the derived category of $G\times \C^\times$-equivariant coherent sheaves on $\mathcal{Z}  $, this is the affine Hecke category for $G$ to be denoted by $\mathcal{AH}$. This category is monoidal with respect to convolution. In particular, we can consider its trace category, cf. \cite{BNP}. In \cite[Theorem 1.4.8]{BNP}, the authors identify this trace category with the derived category of $G\times \C^\times$-equivariant coherent sheaves on the derived scheme $\mathfrak{C}  :=\{(g,x)\in G\times \g| \operatorname{Ad}(g)x=x\}$ set-theoretically supported on the locus $(g,x)$ with nilpotent $x$. The latter category will be denoted by $\operatorname{Coh}_{\mathcal{N},G\times \C^\times}(\mathfrak{C}  )$. The trace functor $\mathcal{AH}\rightarrow \operatorname{Coh}_{\mathcal{N},G\times \C^\times}(\mathfrak{C}  )$ is the usual pull-push functor between coherent sheaves on cotangent bundles using the correspondence $B\backslash G/ B\leftarrow G/\operatorname{Ad}(B)\rightarrow G/\operatorname{Ad}(G)$, we will recall the construction in more detail in Section \ref{SSS:affine_CH}. Denote this functor by $\widetilde{\operatorname{CH}}$. In this paper we will only use this functor but not the claim of \cite[Theorem 1.4.8]{BNP}.

Assume now that $G=\operatorname{GL}_n$. We decorate the schemes and categories appearing in the previous paragraph for $G=G_n:=\operatorname{GL}_n$ with superscript $n$: $\mathcal{AH}_n, \mathfrak{C}  _n$, etc. For $m,n\in \mathbb{Z}_{\geq 0}$, let $P_{m,n}\subset G_{m+n}$ denote the stabilizer of $\mathbb{C}^m\subset \mathbb{C}^{m+n}$. Thanks to the homomorphisms $G_m\times G_n\twoheadleftarrow P_{m,n}\hookrightarrow G_{m+n}$, we get functors (to be called {\it parabolic induction})
\begin{align*}
    &\mathcal{AH}_m\boxtimes \mathcal{AH}_n\rightarrow \mathcal{AH}_{m+n},\\ &\operatorname{Coh}_{\mathcal{N},G_m\times \C^\times}(\mathfrak{C}_m  )\boxtimes \operatorname{Coh}_{\mathcal{N},G_n\times \C^\times}(\mathfrak{C}_n  )\rightarrow \operatorname{Coh}_{\mathcal{N},G_{n+m}\times \C^\times}(\mathfrak{C}_{n+m}  ).
\end{align*}
It is not difficult to see (cf. Section \ref{S_pind_compat}) that the parabolic induction functors intertwine $\widetilde{\operatorname{CH}}_m\boxtimes \widetilde{\operatorname{CH}}_n$ with $\widetilde{\operatorname{CH}}_{m+n}$. Both $$\bigoplus_{n\geq 0}\operatorname{Tr}(\mathcal{AH}_n), \bigoplus_{n\geq 0}\operatorname{Coh}_{\mathcal{N},G_n\times \C^\times}(\mathfrak{C}_n  )$$ acquire monoidal structure via parabolic induction and $\bigoplus_{n\geq 0}\widetilde{\operatorname{CH}}_n$ becomes a monoidal functor.

Now we proceed to describing the main result of our paper. Recall that the elliptic Hall algebra comes with distinguished generators labeled by pairs of integers $(r,a)\neq (0,0)$. For the positive half we take $r>0$. A natural question is to determine natural categorifications of these elements and find their preimages  under the functor $\widetilde{\operatorname{CH}}_r$. A closely related question was asked by Gorsky and Negut, \cite{GN1,GN2} in a somewhat different setting, we'll discuss a relation between their work and ours below in this section and in Section \ref{SS:Lie algebra case}.

Note that the subgroup of length $0$ elements in the affine Weyl group for $\operatorname{GL}_r$ is identified with $\mathbb{Z}$, let $\varpi$ denote a generator. An expected categorification of the generator corresponding to $(r,a)$ is $\widetilde{\operatorname{CH}}(\varpi^a)$ at least when $r,a$ are coprime. The goal of this paper is to compute this coherent sheaf on the commuting scheme $\mathfrak{C}_r$. Let us explain this result here.

Assume that $r$ and $a$ are coprime. One has the unique irreducible twisted $\operatorname{GL}_{r}$-equivariant irreducible D-module on the unipotent cone of $\operatorname{PGL}_r$, where the twist is by the character $\frac{a}{r}\operatorname{tr}$ of $\mathfrak{gl}_r$. Let $\mathbb{N}_{a,r}$ denote the pullback of this D-module to $\operatorname{GL}_r$. The D-module  $\mathbb{N}_{a,r}$ comes with a mixed Hodge structure, hence with a distinguished choice of a good filtration. Taking the associated graded we, in fact, get a $\operatorname{GL}_r\times \C^\times$-equivariant coherent sheaf on $\mathfrak{C}_r  $. We show that this equivariant sheaf is isomorphic to $\widetilde{\operatorname{CH}}(\varpi^a)$. The proof exploits the compatibility between the operation of taking the Hodge associated graded and the functor $\widetilde{\operatorname{CH}}$. Namely, we identify an object in a suitable finite Hecke category whose Hodge associated graded sheaf coincides with $\varpi^a$, and show that the finite analog of the character functor maps this object essentially to $\mathbb{N}_{a,r}$.

In the preceding discussion we assumed that $r,a$ are coprime. We have not computed the image of $\varpi^a$ for general $a$. However, from some perspective, this is not a natural question. Instead, one uses the following construction. Let $d=\operatorname{GCD}(r,a), r_0=r/d, a_0=a/d$, so that $(r_0,a_0)$ are coprime. Instead of considering $\widetilde{\operatorname{CH}}(\varpi^a)$, one constructs an action of $S_d$ on $\widetilde{\operatorname{CH}}(\varpi^{a_0})^{d}$, where the power is for the monoidal structure coming from parabolic induction. Then one looks at the $S_d$-isotypic components, which categorify some generators of the positive half, while the standard generators are expected to be categorified by a complex of the isotypic components.
In the main body of the paper we indeed construct an $S_d$-action, again using a connection to Hodge D-modules.

\begin{rem}
It is expected, \cite[Conjecture 1.5.1]{LNY} that $\operatorname{Tr}(\mathcal{AH})$ is also equivalent to the category of constructible sheaves with nilpotent singular support on the moduli stack of $\check{G}$-bundles on an elliptic curve. Combining this with the equivalence  $\operatorname{Tr}(\mathcal{AH})\cong \operatorname{Coh}_{\mathcal{N},G\times \C^\times}(\mathfrak{C}  )$ from \cite{BNP}, one gets a version of the geometric Langlands correspondence for elliptic curves. We expect that on the constructible side, the object $\varpi^a\in \mathcal{AH}_n$ with $\operatorname{GCD}(a,n)=1$ map to the extensions by $0$ of trivial local systems on the stratum of  stable bundles of rank $n$ and degree $a$.
%
\end{rem}

\begin{rem}
The elliptic Hall category in its coherent realization acts via Hall product on the derived category of coherent sheaves on the Hilbert scheme of points on the plane. In a future work we plan to use this action to study noncommutative resolutions of the configuration space of points on the plane derived equivalent to the Hilbert scheme of points, and their constructible realization via elliptic mirabolic character sheaves, realizing the program described in \cite{Ok}, \cite[Conjecture 1]{ABM} in this special case. The sheaves studied in the present paper play a key role in that program, this was one of our sources of motivation. The noncommutative resolutions are also closely related to modular representations of rational Cherednik algebras studied in \cite{BezLosev}.
\end{rem}

Finally, we would like to discuss a connection between our work and \cite{GN1,GN2}. The purpose of these papers is to relate the cocenters of the affine Hecke categories for $\operatorname{GL}_n$ to the category of $(\C^\times)^2$-equivariant coherent sheaves on the additive version of the commuting stacks: $\{(x,y)\in \mathfrak{gl}_n| [x,y]=0\}/\operatorname{GL}_n$. Here the two copies of $(\C^\times)^2$ act by rescaling $x,y$ individually. This, still conjectural, realization has the advantage to directly decategorify (by passing to the Grothendieck groups) the positive half of the elliptic Hall algebra that is missing in our approach. On the other hand, our approach is more natural from the geometric and representation theoretic perspectives. It is also worth mentioning that our answer for $\widetilde{\operatorname{CH}}(\varpi^a)$ has a direct additive analog.
By \cite[Proposition 3.8]{ma}, in the coprime case the additive analogs of $\widetilde{\operatorname{CH}}(\varpi^a)$ defined as Hodge associated graded of certain twisted equivariant D-modules on $\mathfrak{gl}_r$ coincide with the sheaves $\mathcal{H}'_{a,r}$ from  \cite[Section 2.3]{GN2}.

The paper is organized as follows. In Section \ref{S_preliminaries} we gather various preliminaries: on equivariant mixed Hodge modules, their associated graded sheaves, monodromic Hecke categories, and the character functors for mixed Hodge modules and for coherent sheaves. The results and constructions in this section are versions of known results and constructions. Starting Section \ref{SS_assoc_graded} we present new results. In Section \ref{SS_assoc_graded} we explain how to compute the Hodge associated graded sheaves of standard and costandard objects in monodromic Hecke categories, this is our key tool in computing the objects $\widetilde{CH}(\varpi^a)$. The computation of these objects is carried in Section \ref{S_traces}, the main result is Theorem \ref{thm: coxeter_cuspidal} there. In Section \ref{S_pind_compat} we study the monoidal structure on the positive half of the elliptic Hall category coming from parabolic induction. In Section \ref{S_sym_group_act} we introduce the symmetric group action on the $d$-th power of $\widetilde{\CH}(\varpi^a)$ mentioned above.

{\bf Acknowledgments}. R.B was partially supported by the NSF grant DMS-2101507. I.L. was partially supported by the NSF grant DMS-2501558.  X.M. was supported by the NSF grant DMS-2503504. We would like to thank Dougal Davis for stimulating discussions including the arguments in Section \ref{SS_Hodge_assoc_gr}. We also thank Sam Raskin for stimulating discussions and Eugene Gorsky for comments on the previous version of the manuscript.

ChatGPT was used for auditing proofs and improving exposition. Besides cosmetic editing, it pointed out pullback conventional error in Proposition \ref{prop:gr_pushforward}, shift errors in Lemma \ref{lem:BS_conormal}  and Theorem \ref{thm: coxeter_cuspidal}, ordering error in Lemma \ref{lem:BS_conormal} and sign errors in Corollary \ref{coro:gr length 0}, suggested a simplification of the proof of Theorem \ref{thm: coxeter_cuspidal} and brought \cite{NTY} to our attention.

It is an honor for us to dedicate this paper to Victor Ginzburg whose insights have had a deep influence on the subject and on the authors.

\section{Preliminaries}\label{S_preliminaries}

\subsection{Equivariant mixed Hodge modules and their associated graded}\label{SS_Hodge_assoc_gr}
In Saito's theory \cite{saitofrench, saito}, the structure of a (polarizable) mixed Hodge module on a smooth algebraic variety $X$ consists of a regular holonomic left $D_X$-module, a weight filtration, a good filtration on $M$ called the Hodge filtration, a polarization, and a $\Q$-structure satisfying compatibility conditions. We use the complex version in the sense of Sabbah and Schnell \cite{SabbahSchnell}, in which the $\Q$-structure is omitted. Polarizations will not play a role for our purposes either.

All functors below are derived unless otherwise specified. Let $f: X\to Y$ be a morphism between two smooth varieties, we reserve $f_*$ and $f^*$ for pushforward and pullback functors for quasi-coherent sheaves. Also, denote $!$ pullback  of quasi-coherent sheaves  by $f^{!,\qc}$.

Suppose the smooth variety $X$ has an action of an affine algebraic group $G$. Let $D^b_G\MHM(X)$ be the equivariant dg category of mixed Hodge modules on $X$. We first recall the definition. Notice that we have $n+2$ morphisms $G^{n+1}\times X\rightarrow G^n\times X$: $$(g_0,\ldots,g_n,x)\mapsto (g_1,\ldots,g_n,x), (g_0g_1,g_2,\ldots,g_n x),\ldots, (g_0,\ldots,g_{n-1},g_nx)$$ for all $n\geqslant 0$. They allow us to form the inverse system of dg categories by pullback:
\begin{equation}
    \label{eq:pullback_diagram} D^b\MHM(X)\rightrightarrows D^b\MHM(G\times X)\mathrel{\substack{\longrightarrow\\[-.6ex]\longrightarrow\\[-.6ex]\longrightarrow}} D^b \MHM(G^2\times X)\ldots
\end{equation}
One defines the dg category $D^b_G\MHM(X)$ as the inverse limit of this diagram. Since the morphisms in (\ref{eq:pullback_diagram}) are t-exact, the category $D^b_G\MHM(X)$ acquires a t-structure whose heart is the abelian category $\operatorname{MHM}_G(X)$ of $G$-equivariant mixed Hodge modules. The Hodge filtration on any object in $\operatorname{MHM}_G(X)$ is $G$-stable.

The equivariant mixed Hodge derived categories admit the six functor formalism. Assume  $f: X\to Y$ is $G$-equivariant. We use $f_\dagger$, $f_!$, $f^\dagger$ and $f^!$ to denote the (proper) pushforward and (proper) pullback between $D^b_G\MHM(X)$ and $D^b_G\MHM(Y)$.
Under the forgetful functor, these functors are compatible with the corresponding ones on the complexes of $D$-modules: let $\Dcal_{X\to Y}$ and $\Dcal_{Y\leftarrow X}$ be the transfer $D$-modules as defined in \cite[1.3]{HTT}. We then have \[
    f_\dagger:=f_*(\Dcal_{Y\leftarrow X}\otimes_{\Dcal_X} -),\quad  f^\dagger=\Dcal_{X\to Y}\otimes_{f^*\Dcal_Y} f^*(-) ,
\]
Define $f_!=\Dbb_Y^{-1}\circ f_\dagger\circ \Dbb_X$ and $f^!=\Dbb_X^{-1}\circ f^\dagger\circ \Dbb_Y$. Here $\Dbb_?=\Hom_{\Dcal_?}(-,\Dcal_?)[\dim(?)]$ is the Verdier duality for $?=X$ or $Y$.

One has the the functor $\gr: D^b\MHM(X)\rightarrow D^b \Coh^{\mathbb{G}_m}(T^*X)$ via taking associated graded with respect to the Hodge filtration. It is t-exact as Hodge filtrations are strict according to \cite[Proposition 5.1.14]{saitofrench}. We now want to explain its version in the $G$-equivariant setting. Namely, we can consider the smooth quotient stack $X/G$ and the derived stack $T^*(X/G)$. One can talk about the dg category of $\mathbb{G}_m$-equivariant quasi-coherent sheaves with bounded coherent cohomology, which we denote by $D^b\Coh_{\mathbb{G}_m}(T^*(X/G))$. Its heart is the category of $G\times \mathbb{G}_m$-equivariant coherent sheaves on $\mu^{-1}(0)$, where $\mu:T^*X\rightarrow \g^*$ is the moment map.  Every object in $\MHM_G(X)$ is a strongly $G$-equivariant D-module. Since the Hodge filtration is $G$-stable, we see that $\gr \mathcal{F}$ is scheme-theoretically supported on $\mu^{-1}(0)$ and is a $G\times \mathbb{G}_m$-equivariant coherent sheaf on $\mu^{-1}(0)$. This gives an exact functor between abelian categories $\MHM_G(X)\rightarrow \Coh_{G\times \mathbb{G}_m}(\mu^{-1}(0))$. Note that the latter is the heart of the natural t-structure on the dg category $D^b\Coh_{\mathbb{G}_m}(T^*(X/G))$.

\begin{prop}\label{Prop:Hodge_gr_equiv}
The functor of taking Hodge associated graded extends to $D^b_G\MHM(X)\rightarrow D^b\Coh_{\mathbb{G}_m}(T^*(X/G))$.
\end{prop}
\begin{proof}
The proof is in several steps.

{\it Step 1.} For a smooth (finite type) algebraic stack $X$ with affine stabilizers, one can consider the Dolbeault stack $X_{Dol}$. It is defined as the relative classifying stack of the formal group scheme $\operatorname{exp}(T_X)$ over $X$, where $T_X$ is the tangent complex of $X$ with its abelian (dg) Lie algebra structure and $\exp(T_X)$ is the sheaf of formal groups associated to this sheaf. In the notation of \cite[Ch. 7, \S1.4.1]{GR}, this is $\operatorname{Vect}_X(T_X[1])$.

Its crucial property is that the (dg) category $\operatorname{IndCoh}^{\mathbb{G}_m}(X_{Dol})$ of $\mathbb{G}_m$-equivariant ind-coherent sheaves on $X_{Dol}$ is equivalent to the (dg) category $\operatorname{QCoh}^{\mathbb{G}_m}(T^*X)$ of $\mathbb{G}_m$-equivariant quasi-coherent sheaves on $T^*X$, see \cite[Theorem A]{Chen}. Moreover, let $f:X\rightarrow Y$ be a morphism of smooth stacks. Consider the maps
\begin{equation}
    \label{diagram_lagrangian} T^*X \xleftarrow{\rho_f}X\times_Y T^*Y\xrightarrow{\varpi_f}T^*Y.
\end{equation}
where $\rho_f$ is the co-differential of $f$ and $\varpi_f$ is the natural projection. When $f$ is smooth and representable, $\varpi_f$ is also smooth and $\rho_f$ is a closed embedding. These morphisms yields a functor $\rho_{f,*}\varpi_{f}^*: \operatorname{QCoh}^{\mathbb{G}_m}(T^*Y)\rightarrow \operatorname{QCoh}^{\mathbb{G}_m}(T^*X)$. On the other hand, $f$ gives rise to $f_{Dol}:X_{Dol}\rightarrow Y_{Dol}$ and one can consider the $!$-pullback functor $\operatorname{IndCoh}^{\mathbb{G}_m}(Y_{Dol})\rightarrow \operatorname{IndCoh}^{\mathbb{G}_m}(X_{Dol})$. A part of \cite[Theorem A]{Chen}, {cf. \cite[(6.2.1)]{FR}}, is that under the equivalences $\operatorname{IndCoh}^{\mathbb{G}_m}(?_{Dol})\cong \operatorname{QCoh}^{\mathbb{G}_m}(T^*?)$ (where $?=X,Y$), the functor $f^!_{Dol}$ becomes $\rho_{f,*}\varpi_{f}^*$, {which is compatible with compositions of morphisms.}

{\it Step 2}. We'll need the following classical result about about the compatibility between taking the Hodge associated graded sheaf and smooth pullbacks. Let $f:X\to Y$ be a smooth morphism. Using the construction of Step 1 we get a functor $\gr: D^b\MHM(?)\rightarrow \operatorname{IndCoh}^{\mathbb{G}_m}(?_{Dol})$ for $?=X$ or $Y$. Take $(M,F)\in \MHM(Y)$, { where $F$ denotes the Hodge filtration}. By \cite[Theorem 30.1]{schnell} (also c.f \cite[Theorem 4.7]{kashiwara2003}), we have $F_k(f^\dagger M)=f^*F_kM$. Hence flatness of $f$ implies $\gr^F_k(f^\dagger M)=f^*\gr^F_kM$. Therefore we have a natural isomorphism
\begin{align}
    \label{iso: alpha} \alpha: \gr\circ f^\dagger \xrightarrow{\sim}\rho_{f,*}\varpi_{f}^*\circ \gr.
\end{align}
Moreover, since $f^\dagger$, $\gr$ and $\rho_{f,*}\varpi_{f}^*$ are all exact, and the isomorphisms above are compatible with compositions of smooth morphisms, we actually have a natural isomorphism of functors between the bounded dg categories, compatible with compositions. In other words, the following diagram commutes and the commutativity is compatible with compositions of smooth morphisms.
	\begin{align}
    \label{diag: pullback gr} \xymatrix{D^b\MHM(Y) \ar[d]^{\gr}\ar[rrrr]^ - {f^\dagger} &&&&D^b\MHM(X)\ar[d]^{\gr}\\ \operatorname{IndCoh}_{\mathbb{G}_m}(Y_{Dol})\ar[rrrr]^ - {f^!_{Dol}} &&&&\operatorname{IndCoh}_{\mathbb{G}_m}(X_{Dol}). }
\end{align}

{\it Step 3}. Apply Step 2 to the smooth maps in \eqref{eq:pullback_diagram} to form a commutative diagram in the shape of \eqref{diag: pullback gr} going infinitely to the right. Since $\alpha$ is compatible with compositions, we can take the inverse limit of this diagram. Note that \begin{equation}
    \label{eq:descent} \operatorname{IndCoh}_{\mathbb{G}_m}((X/G)_{Dol})\xrightarrow{\sim}\operatorname{lim}_n [\operatorname{IndCoh}_{\mathbb{G}_m}((G^n\times X)_{Dol})]
\end{equation}
by a special case of a more general descent statement proved by \cite[Proposition 2.3.8]{Chen}. Thanks to the construction of $D^b_G\MHM(X)$, we hence obtain a functor
\[
    D_G^b\MHM(X)\to \Ind\Coh_{\Gbb_m}((X/G)_{Dol})\cong \mathrm{QCoh}^{\Gbb_m}(T^*(X/G)).
\]
Finally we notice that the image of $\gr$ lands inside the bounded part $D^b\Coh_{\mathbb{G}_m}(T^*(X/G))$.
\end{proof}

\subsection{Associated graded of direct/inverse images and inductions}\label{SS_assoc_graded_general}


For any $G$-equivariant morphism between smooth $G$-varieties $X\rightarrow Y$ we have the direct analog of (\ref{diagram_lagrangian}), where we keep the same notations when there is no risk of confusion:
\begin{equation}
    \label{eq:lagrangian_stack} T^*(X/G) \xleftarrow{\rho_f}X/G\times_{Y/G} T^*(Y/G)\xrightarrow{\varpi_f}T^*(Y/G).
\end{equation}
\subsubsection{Inverse image} Write $\delta_f:=(\rho_f)_* \varpi_f^*$.
\begin{lem}
\label{lemma_gr_pullback}
Suppose $f:X\to Y$ is a $G$-equivariant smooth morphism. The following diagram commutes.
	\[
    \xymatrix{D^b_G\MHM(Y) \ar[d]^{\gr}\ar[rrrr]^ - {f^\dagger} &&&&D^b_G\MHM(X)\ar[d]^{\gr}\\ \Coh_{\mathbb{G}_m}(T^*(Y/G))\ar[rrrr]^ - {\delta_f} &&&&\Coh_{\mathbb{G}_m}(T^*(X/G)) }
\]
\end{lem}
\begin{proof}
In the case of $\mathrm{id}\times f: G^n\times X\to G^n\times Y$ for all $n\geq 0$, we have a natural isomorphism $\alpha_n$ and hence a commutative diagram from Step 2 of the proof of Proposition \ref{Prop:Hodge_gr_equiv}. As $\alpha_n$'s are compatible with the compositions, these commutative diagrams again assemble to an infinite commutative diagram using \eqref{eq:pullback_diagram}. Now take limit and note the image lands in the desired bounded categories.
\end{proof}
\subsubsection{Direct image}
Let $\wt{\omega}_{X/Y}$ denote the inverse image of the relative canonical bundle $\omega_{X/Y}$ to $X\times_Y T^*Y$. {For a $\C^\times$-equivariant sheaf $\Fcal$, or equivalently a grading $\Fcal=\bigoplus_{k\in \Z} \Fcal_k$, let $\Fcal\langle \ell\rangle $ denotes the shift such that $(\Fcal\langle \ell\rangle)_k=\Fcal_{k-\ell}$.}

Let
	\[
    \gamma_f(-):=(\varpi_f)_*\bigl(\rho_f^*(-)\otimes\wt\omega_{X/Y}\langle\dim(Y)-\dim(X)\rangle \bigr).
\]
that we can view as a functor $\operatorname{QCoh}(T^*X)\rightarrow \operatorname{QCoh}(T^*Y)$ (the  derived categories of quasi-coherent sheaves) or as a functor $\operatorname{QCoh}(T^*(X/G))\rightarrow \operatorname{QCoh}(T^*(Y/G))$, depending on the context. Note that if $f$ is proper, then $\gamma_f$ restricts to the coherent subcategories (i.e., complexes with bounded coherent cohomology).

\begin{prop}\label{prop:gr_pushforward}
When $p: X\to Y$ is a $G$-equivariant projective morphism, the following diagram commutes.
	\[
    \xymatrix{D^b_G\MHM(X) \ar[d]^{\gr}\ar[rrrr]^ - {p_\dagger} &&&&D^b_G\MHM(Y)\ar[d]^{\gr}\\ \Coh_{\mathbb{G}_m}(T^*(X/G))\ar[rrrr]^ - {\gamma_p} &&&&\Coh_{\mathbb{G}_m}(T^*(Y/G)) }
\]
\end{prop}
\begin{proof}
\newcommand{\dg}{\mathrm{dg}}
Below, we will use subscripts $_{\tr}$ and $_{\dg}$ to distinguish between triangulated categories and dg categories.

We first recall the natural isomorphism of functors $\lambda_{\tr}:\gr p_\dagger\to\gamma\gr$ on the level of triangulated categories, from \cite[2.3.6-2.3.7]{saitofrench} and \cite[2.3.2]{laumon}.  
For a complex $M$ of mixed Hodge modules on $X$, there is an isomorphism in the triangulated derived category of graded $\gr(\Dcal_Y)$-modules by strictness of Hodge filtrations:
\begin{align}
    \label{iso: Laumon} \gr(p_\dagger M)=\gr\left(p_*(D_{Y\leftarrow X}\otimes_{\Dcal_X}M)\right)\xrightarrow{\sim} p_*\!\left(\gr(\Dcal_{Y\leftarrow X}) \otimes_{\gr (\Dcal_X)} \gr (M)\right).
\end{align}
Now $\gr(\Dcal_{Y\leftarrow X})$ can be identified with the pushforward of $ \Ocal_{X\times_Y T^*Y}\otimes\omega_{X/Y}\langle\dim Y-\dim X\rangle$ to $X$. Write $q: T^*Y\to Y$ for the projection. Then taking $\Ocal_{T^*Y}\otimes_{q^{-1}\gr(\Dcal_Y)}q^{-1}(-)$ of both sides of \eqref{iso: Laumon}, we obtain $\lambda_\tr:\gr p_\dagger(M)\to \gamma \gr(M).$

Next, decompose $p$ as a composition of a $G$-equivariant closed embedding and a $G$-equivariant projection. For simplification, below we write $\rho:=\rho_?$,  $\varpi:=\varpi_?$ where $?$ could be different maps in different contexts.

\textit{The case of the closed embedding,} $i:X\to Z:=Y\times\mathbb P^d$. Since $\Dcal_{Z\leftarrow X}$ is flat as a right filtered $\Dcal_X$-module, $i_\dagger$ is exact. Moreover, $\gamma$ is also exact. Therefore, $\lambda_\tr$ extends to a natural isomorphism $\lambda_\dg$ between the bounded dg categories.

\if force the pushforward $i_\dagger M=i_*(\Dcal_{Z\leftarrow X}\otimes_{\Dcal_X}M)$ for $M\in \MHM(X)$ satisfies
\[
    \gr(i_\dagger M)\simeq i_*\bigl(\gr\Dcal_{Z\leftarrow X}\otimes_{\gr \Dcal_X}\gr M\bigr) \simeq \varpi_*\bigl(\rho^*\gr M\otimes\wt\omega_{X/Z}\bigr)\langle\dim Z-\dim X\rangle=\gamma_i(\gr M),
\]
using $\gr(\Dcal_{Z\leftarrow X})\simeq\mathcal O_{X\times_Z T^*Z}\otimes\omega_{X/Z}\langle\dim Z-\dim X\rangle$ and the flatness of $\rho$. \fi

We proceed to show that $\lambda_\dg$ is compatible with smooth pullback -- {as in  the proof of Proposition \ref{Prop:Hodge_gr_equiv} this is needed to pass from the non-equivariant to equivariant categories}. Let $f: Z'\to Z$ be smooth, $X'=X\times_Z Z'$, with $f': X'\to X$, $i': X'\to Z'$. Let $\alpha:\gr f^\dagger\xrightarrow{\sim}\delta_f\gr $ be as in Lemma~\ref{lemma_gr_pullback}, $\theta: f^\dagger i_\dagger\xrightarrow{\sim} i'_\dagger f'^\dagger$ the base change, and $\beta:\delta_f\gamma_i\xrightarrow{\sim}\gamma_{i'}\delta_{f'}$ flat base change. We claim that the composition
	\begin{equation}
    \label{eq:long_composition} \delta_f\gamma_i\gr \xrightarrow{\delta_f\lambda^{-1}_\dg}\delta_f\gr i_\dagger \xrightarrow{\alpha^{-1}i_\dagger}\gr f^\dagger i_\dagger \xrightarrow{\gr\theta}\gr i'_\dagger f'^\dagger \xrightarrow{\lambda_\dg f'^\dagger}\gamma_{i'}\gr  f'^\dagger \xrightarrow{\gamma_{i'}\alpha}\gamma_{i'}\delta_{f'}\gr
\end{equation}
equals $\beta\gr$ {(here we abuse the notation and use functors such as $\gr$ to denote the identity endomorphisms)}. For $M\in \MHM(X)$, both (\ref{eq:long_composition}) and $\beta\gr$ are maps of coherent sheaves. By exactness, this case suffices for bounded complexes and may be checked \'{e}tale locally, where $Z=X\times\mathbb C^c$, $Z'=Z\times W$ and $f$ is the projection. Choose normal coordinates $t_1,\dots, t_c$ and trivialize $\omega_{X/Z}$. Let $\xi_{t_j}$ be the cotangent coordinate along $t_j$. For $\nu=(\nu_1,\dots,\nu_c)$, write $\partial_t^\nu=\prod_j\partial_{t_j}^{\nu_j}$ and $\xi_t^\nu=\prod_j\xi_{t_j}^{\nu_j}$. For $m\in F_kM$, let $[m]$ denote its image in $\gr_k^F M$. Then $\lambda_\dg$ sends $[\partial_t^\nu\otimes m]\mapsto\xi_t^\nu\otimes[m]$, $\alpha$ is the identity on sections since $F_k f^\dagger=f^*F_k$, and $\theta$ is $a\otimes\partial_t^\nu\otimes m\mapsto\partial_t^\nu\otimes(a\otimes m)$. Chasing a section $a\otimes\xi_t^\nu\otimes[m]$ under (\ref{eq:long_composition}) gives $\xi_t^\nu\otimes a\otimes[m]$, which is its image under $\beta$.

\textit{The case of the projection} $\pi: Y\times\Pbb^d\to Y$. We will study the adjunction mate $\mu_\dg$ of the natural isomorphism $\alpha$ {from (\ref{iso: alpha})}, which is elaborated on below. Write $\gamma:=\gamma_\pi$. { When for two functors $\mathcal{F},\mathcal{G}$ we write $\mathcal{F}\dashv \mathcal{G}$, we mean that $\mathcal{F}$ is left adjoint to $\mathcal{G}$.} Clearly, $\rho^*\dashv\rho_*$. Since $\varpi$ is proper, we also have $\varpi_*\dashv\varpi^{!,\qc}$. It follows that the right adjoint of $\gamma$ is
\[
    \rho_*\bigl(\varpi^{!,\qc}(-)\otimes\wt{\omega}^{-1}_{Y\times\Pbb^d/Y}\bigr)\langle d\rangle.
\]
For the smooth proper map $\varpi$, $\varpi^{!,\qc}=\varpi^*(-)\otimes\wt{\omega}_{Y\times\Pbb^d/Y}[d]$. As a result, the right adjoint of $\gamma$ is
\[
    \sigma=\rho_*\varpi^*[d]\langle d\rangle.
\]
As $\pi$ is proper, the right adjoint of  $\pi_\dagger$ is $\pi^\sharp:=\pi^\dagger[d](d)$ where $(-)$ denotes the Tate twist. Let
\[
    \eta: \id\to \pi^\sharp\pi_\dagger,\quad c:\gamma\sigma\to \id
\]
be the adjunction unit and counit respectively.  The mate of $\alpha$ is the composition
\[
    \mu_\dg: \gamma \gr \xrightarrow{\gamma \gr(\eta)} \gamma \gr \pi^\sharp\pi_\dagger \xrightarrow{\gamma \alpha\pi_\dagger}\gamma \sigma \gr\pi_\dagger \xrightarrow{c\gr\pi_\dagger}\gr\pi_\dagger.
\]
Since all the natural transformations in the composition are compatible with smooth pullbacks, so is $\mu_\dg$. Therefore, to descend $\mu_\dg$ to obtain the desired isomorphism $\gamma \gr\cong\gr\pi_\dagger$, we just need to show that $\mu_\dg$ is invertible.

Consider $\gr\pi_\dagger\eta:\gr\pi_\dagger\to\gr\pi_\dagger\pi^\sharp\pi_\dagger$ and $\tau\pi_\dagger$, where $\tau$ is given by
\begin{align*}
    \tau: \gr\pi_\dagger\pi^\sharp\xrightarrow{\lambda_\tr{\pi^\sharp}}\gamma\gr\pi^\sharp \xrightarrow{\gamma\alpha}\gamma\sigma\gr\xrightarrow{c{\gr}}\gr.
\end{align*}
Applying naturality of $\lambda_\tr$ to $\eta: \id \to \pi^\sharp\pi_\dagger$, we get
\[
    \lambda_\tr \pi^\sharp\pi_\dagger\circ \gr(\pi_\dagger \eta)=\gamma \gr(\eta)\circ \lambda_\tr.
\]
Postcomposing it with $c \gr \pi_\dagger \circ \gamma \alpha \pi_\dagger $, we obtain $\tau{\pi_\dagger}\circ\gr\pi_\dagger\eta=\mu_\dg\lambda_\tr$. We will show that $\mu_\dg\lambda_\tr=u\cdot\id$ for a nonzero constant $u$.

As $\pi=\id\times a$, where $a$ is the map $\Pbb^d\to\pt$, all maps and transformations above decompose into products. Hence below we may assume $\pi$ is simply the projection $\Pbb^d\to \pt$. Since $\pi_\dagger \pi^\sharp\C\cong\bigoplus_j \Hrm^j(\Pbb^d)(d)[2d-j]$ lies in $D^{\le0}$ with $\Hrm^0=\Hrm^{2d}(\Pbb^d)(d)\cong\C$, and $\gr$ is t-exact, we have
\[
    \Hom(\gr \pi_\dagger \pi^\sharp\C,\C)=\Hom(\Hrm^0\gr \pi_\dagger \pi^\sharp\C,\C)\simeq\C .
\]
It is spanned by $\gr\varepsilon$, where $\varepsilon: \pi_\dagger \pi^\sharp\to\id$ is the counit.
Thus $\tau=u\cdot\gr\varepsilon$ for some constant $u\in\C$. Since $\varepsilon{\pi_\dagger}\circ\pi_\dagger\eta=\id$, we get $\mu_\dg\lambda_\tr=u\cdot\id$, and it remains to show $u\neq0$.

Since both $\alpha$ and $\lambda$ are isomorphisms, it suffices to show that $\Hrm^0(c)$ is an isomorphism, as that would imply $\Hrm^0(\tau)\neq 0$. By Yoneda, for $K\in\Coh_{\Gbb_m}(\pt)^\heartsuit$, the map $\Hrm^0(c)_K: \Hrm^0(\gamma\sigma K)\to K$ is an isomorphism if and only if precomposition
\[
    \Hrm^0(c)_K^*:\Hom(K,K')\to\Hom(\Hrm^0(\gamma\sigma K),K')
\]
is bijective for every $K'\in\Coh_{\Gbb_m}(\pt)^\heartsuit$. Now we have
\begin{align*}
    \Hom(\gamma \sigma K,K')&\cong \Hom(\sigma K,\sigma K')\\
    &\cong \Hom(\varpi^*K,\varpi^*K')\quad \because \rho \text{ is a closed embedding}\\ &\cong\Hom(K,\varpi_*\varpi^*K')\cong\Hom(K,K') \because \varpi_*\Ocal_{\Pbb^d}=\C.
\end{align*}
Denote the composition by $\Psi$. Let $u_K: K\to \varpi_*\varpi^*K$ be the unit. For $h\in \Hom(K,K')$, the first isomorphism sends $c_K^*(h)$ to $\sigma(h)$. The next equality sends $\sigma(h)$ to $\varpi^*(h)$. Next, $\varpi^*(h)$ is sent to $u_{K'}^{-1}\circ \varpi_*\varpi^*(h)\circ u_K$. By naturality of the unit, we have
\[
    \varpi_*\varpi^*(h)\circ u_K=u_{K'}\circ h.
\]
Consequently, $\Psi(h\circ c_K)=h$. Hence $\Psi$ is the inverse of $c^*$. Moreover,
\[
    \gamma\sigma(\C)=\varpi_*\bigl(\rho^*\rho_*\varpi^*\C\otimes\wt\omega_{\Pbb^d}\bigr)[d]
    =\bigoplus_{i=0}^d\varpi_*\bigl(\wedge^{d-i}\Omega_{\Pbb^d}\bigr)[d+i]=\bigoplus_{i=0}^d\C[2i]
\]
(up to $\Gbb_m$-grading shifts), where the second equality uses 
that $\rho^*\rho_*\Ocal=\bigoplus_j\wedge^jT_{\Pbb^d}[j]$. Hence $\gamma\sigma K\in D^{\le0}$ and $$\Hom(\gamma\sigma K,K')=\Hom(\Hrm^0(\gamma\sigma K),K').$$ Therefore $\Hrm^0(c)$ is an isomorphism, $u\neq0$, and $\mu_\dg$ is invertible.
\end{proof}
\subsubsection{Induction}
Let $H$ be a closed subgroup of $G$, which acts on the product $G\times X$ by $h\cdot(g,x)=(gh^{-1},hx)$. Let $G\times^H X$ be the quotient of this action. Then we have an isomorphism between quotient stacks $X/H\cong  (G\times^H X)/G$ and $T^*(X/H)\cong T^*((G\times^H X)/G)$.

Let $\mu_H:T^*X\longrightarrow\operatorname{Lie}(H)^*$ and $\mu_G:T^*(G\times^H X)\longrightarrow\operatorname{Lie}(G)^*$. Note that \[
    T^*(X/H)\cong \mu^{-1}(0)^R/H,
\] the derived $0$-fiber of $\mu_H$, and similarly for $G\times^H X$.  Note that $\mu_G^{-1}(0)^R/G\cong \mu_H^{-1}(0)^R/H$. { The following lemma is straightforward.}
\begin{lem}\label{prop_gr_induction}
We have the following commutative diagram.
	\[
    \xymatrix{ D^b_H\MHM(X)\ar[d]^{{\gr}}\ar[rrrr]^{\sim} &&&&D^b_G\MHM(G\times^H X)\ar[d]^{{\gr}}\\ \Coh_{H\times \mathbb{G}_m}(\mu_H^{-1}(0)^R)\ar[rrrr]^{\sim} &&&&\Coh_{G\times \mathbb{G}_m}(\mu_G^{-1}(0)^R). }
\]
\end{lem}

\subsection{Monodromic Hecke categories}\label{SS_Hecke_monodromic}
Let $G$ be a complex connected reductive group with Lie algebra $\gfrak$. Inside $G$, we fix a Cartan subgroup and a Borel subgroup $T\subset B$, with Lie algebras $\tfrak\subset \bfrak$. Let $U=[B,B]$ be the unipotent radical of $B$ with Lie algebra $\nfrak$. Denote the flag variety $G/B$ by $\Bcal$ and the basic affine space $G/U$ by $\wt{\Bcal}$.

Consider the right diagonal action of $T$ on $\wt{\Bcal}\times \wt{\Bcal}$ and define $Y:=(\wt{\Bcal}\times \wt{\Bcal})/T$. On $Y$ we have the diagonal left $G$-action and a $T$-action by $t(x,y)=(x,yt^{-1})$. The fibration $\wt{\Bcal}^2\to Y$ is $G$-equivariant.

We follow the definitions in \cite{LY}. Let $m:T\times T\to T$ be the multiplication map. A rank one character $D$-module on $T$ is a rank $1$ regular integrable connection equipped with an isomorphism $m^\dagger \Lcal\cong \Lcal\boxtimes \Lcal$ and a trivialization of the stalk at $e$. The isomorphism classes of all rank one character $D$-modules form a group, which we denote by $Ch(T)$. Suppose $\Lcal$ is of finite order, or equivalently, the character $\lambda\in \tfrak^*$ that determines the integrable connection is rational. Then there is a finite \'{e}tale cover
\[
    \alpha:\wt{T}\to T
\]
such that $\Lcal=(\alpha_\dagger \Ocal_{\wt{T}})^\theta$ with respect to $\theta=e^{2\pi i \lambda}:\Ker(\alpha)\to \C^\times$. There is an action of $W$ on $Ch(T)$ given by $w\Lcal=(w^{-1})^*\Lcal$. Let $W_\theta\subset W$ be the stabilizer of $\Lcal$, or equivalently the stabilizer of $\theta$.

Suppose a smooth variety $X$ has a $T$ action, and in particular a $\wt{T}$ action through the cover $\alpha$. Under the action of $\Ker(\alpha)$, $D^b_{\wt{T}}\MHM(X)$ decomposes into a direct sum of triangulated subcategories on which $\Ker(\alpha)$ acts by characters. Let $D^b_{\theta}\MHM(X)$ denote the direct summand associated to $\theta$. When $\theta=1$, we recover $D^b_{T}\MHM(X)$.

We further impose $G$-equivariance and set $\Hcal_\theta:=D^b_{G,\theta}\MHM(Y)$; its heart is denoted by $\Hcal_\theta^\heartsuit$.


Let $\pi:B\to B/U\cong T$ be the projection. On $\wt{\Bcal}$, we have the adjoint $B$-action given by $b\cdot g=bg\pi(b)^{-1}$. Consider the quotient $G\times^B\wt{\Bcal}$ of $G\times\wt{\Bcal}$ by the diagonal $B$-action:
\[
    b(g,x)=(gb^{-1},bxb^{-1}),g\in G,x\in\wt{\Bcal}.
\]
On $G\times^B\wt{\Bcal}$, we have a $G\times T$-action defined by $(h,t)(g,x)=(hg,xt^{-1})$. Suppose $g\in G,x\in \wt{\Bcal}$. The map \[
    G\times^B\wt{\Bcal}\xrightarrow{\sim} Y,\quad (g,x)\mapsto (g,gx)
\] defines a $G\times T$-equivariant isomorphism.

The $G\times T$-orbits in $Y$ are in bijection with the $B\times T$-orbits in $\wt{\Bcal}$ and the Weyl group $W$: for every $w\in W$ denote its lift in $N(T)$ still by $w$. Write $\Bcal_w=BwB/B$ and $\wt{\Bcal}_w=BwB/U$. 

Let $U_-$ be the opposite of $U$. Take $U_w=U\cap w\cdot U_-$. The embedding $U_w w T \to \wt{\Bcal}_w$ is an isomorphism, and so finite-rank regular integrable connections on $\wt{\Bcal}_w$ are in bijection with finite-dimensional representations of $\pi_1(T)$.

Let $\Gamma_w=\{(t,w^{-1}(t)t^{-1})\in T\times T\mid t\in T\}$ and $U^w=U\cap w\cdot U$. For any point $x\in \wt{\Bcal}_w$, its stabilizer in $\Ad(B)\times T$ is isomorphic to $\Gamma_w\ltimes U^w$.

Let $i_w:\{w\}\hookrightarrow\wt{\Bcal}_w$. Similar to \cite[Lemma 2.10]{LY}, we have the following lemma.
\begin{lem}\label{lem: mono}
The category $D^b\MHM_{B,\theta}(\wt{\Bcal}_w)$ is nonzero if and only if $w\in W_\theta$. In this case, we have an equivalence
	\[
    i_w^\dagger:\quad  D^b_{B,\theta}\MHM(\wt{\Bcal}_w)\xrightarrow{\sim} D^b_{\Gamma_w}\MHM(\{w\}).
\]
\end{lem}
Let \begin{align}
    \label{defn: loc sys L w theta} \Lbb_{w,\theta}\in D^b_{B,\theta}\MHM(\wt{\Bcal}_w)
\end{align}
be the rank-$1$ object corresponding to the Hodge structure $\C_{\{w\}}$ of type $(0,0)$. 
\begin{align}
    \nabla^e_{w,\theta}&=j_\dagger \Lbb_{w,\theta}  &\nabla_{w,\theta} &= \Ind_B^G \nabla^e_{w,\theta};\label{defn:nabla}\\ \Delta^e_{w,\theta}&=j_! \Lbb_{w,\theta} &\Delta_{w,\theta} &= \Ind_B^G \Delta^e_{w,\theta}; \label{defn:Delta}\\ IC^e_{w,\theta}&=j_{!\dagger} \Lbb_{w,\theta} &IC_{w,\theta} &= \Ind_B^G IC^e_{w,\theta}; \nonumber
\end{align}
where the objects in the second column are in $\Hcal_\theta^\heartsuit$. We call the objects $\nabla_{w,\theta}$ (resp., $\Delta_{w,\theta}$) costandard (resp., standard) objects. Their underlying $D$-modules are the standard and costandard objects of the corresponding highest weight category of equivariant monodromic regular holonomic $D$-modules, indexed by $W_\theta$.

\subsection{Character sheaves}
There is the horocycle correspondence:
\begin{align}
    \label{diagram_horocycle} G\xleftarrow{p} G\times \Bcal\xrightarrow{a}  Y
\end{align}
where $a:(g,h)\mapsto\ol{(\wt{h},g\wt{h})}$ with $\wt{h}$ being a lift of $h\in \Bcal$ in $\wt{\Bcal}$ and $p:(g,h)\mapsto g$.

Consider the functor \[
    \CH:=p_\dagger a^\dagger: \Hcal_\theta\to D^b_G\MHM(G).
\] %
Define $\CS_\theta^\heartsuit\subset\MHM_G(G)$ as the smallest Serre subcategory containing the objects $\Hcal^i\CH(\Fcal)$, where $\Fcal\in\Hcal_\theta^\heartsuit$ and $i\in\Z$. Further, let $\CS_\theta\subset D^b_G\MHM(G)$ denote the full subcategory consisting of complexes all of whose cohomology objects belong to $\CS_\theta^\heartsuit$. Since $\CS_\theta^\heartsuit$ is a Serre subcategory, the long exact cohomology sequence shows that $\CS_\theta$ is triangulated; its definition in terms of cohomology objects also makes it stable under the standard truncation functors. Its heart is precisely $\CS_\theta^\heartsuit$.
\begin{rem}\label{rem:twisted_CH}
Applying these results to the elliptic Hall category {(where the relevant groups are $G=\operatorname{GL}_n$)} requires a monodromic setting. Let $\xi\in  \gfrak^*$ be a rational character. Choose a central isogeny $q:\wt{G}\to G$ such that $q^*\xi$ is the differential of a character $\chi$ of $\wt{G}$. Then one can consider the monodromic derived category $D^b_{\xi,\chi}(Y)$, which is the $\chi$-isotypic component of $\wt{G}$-equivariant derived category. For $G=\GL_n$ and $\xi=\frac{d}{n}\tr$. Take the $n$-cover $q: \wt{G}=\SL_n\times \C^\times\to \GL_n$ such that $q^*\xi$ is the differential of $\chi(g,t)=t^d$. The horocycle correspondence \eqref{diagram_horocycle}, induction as well as the functors $\CH$ and $\HC$ are adaptable in this setting.
\end{rem}
\subsection{Functors $\gr_{\Hcal,\lambda}$ and $\gr_\CS$}

Fix a nondegenerate invariant bilinear form on $\gfrak$ and use it to identify $\gfrak\cong \gfrak^*$.
Use $\cdot$ to denote the adjoint action. Consider the moment map $\mu_G:T^*G\cong G\times \gfrak\to \gfrak$, $(g,x)\mapsto g\cdot x-x$. Define the derived commuting scheme $\Cfrak$ to be the derived zero fiber of $\mu_G:T^*G\cong G\times\gfrak\to\gfrak$. Its classical truncation is $t_0\Cfrak=\{(g,x)\in G\times\gfrak\mid g\cdot x=x\}$. 
The $G\times\C^\times$ action on $G\times \gfrak$:
\[
    (g,z)(h,x)=(ghg^{-1},zg\cdot x),\qquad g,h\in G,\ z\in\C^\times,\ x\in\gfrak,
\]
restricts to a $G\times\C^\times$ action on $\Cfrak$.

Moreover, 
\[
    T^*Y\cong \{(x,y,\alpha,\beta)\in Y\times \gfrak\times \gfrak|x^{-1}\cdot\alpha, y^{-1}\cdot \beta\in \bfrak,  x^{-1}\cdot \alpha\equiv y^{-1}\cdot \beta \text{ mod } \nfrak \}.
\]

Consider the moment map:
\begin{align*}
    \mu_Y:T^*Y \to \gfrak\times\tfrak, (x,y,\alpha,\beta)\mapsto (\alpha-\beta,y^{-1}\cdot\beta \text{ mod }\nfrak).
\end{align*}
Let $\wt{\Ncal}:=G\times^B \nfrak$. The derived zero fiber $\wt{\Zcal}^R$ of $\mu_Y$ is isomorphic to a $T$-torsor over $\Zcal^R:=\wt{\Ncal}\times^R_\gfrak \wt{\Ncal}$, pullback from $Y\to \Bcal^2$. The $G\times\C^\times$-action on $\wt{\Ncal}$:
\[
    (g,z)(\bfrak,x)=(g\cdot \bfrak,zg\cdot x)
\] induces a $G\times\C^\times$-action on $\Zcal^R$.

For a derived scheme $X$ with an action of an affine algebraic group $H$, let $\Coh_H(X)$ denote the full subcategory of the quasi-coherent derived category consisting of complexes with bounded coherent cohomology. Let $\Coh_H(X)^\heartsuit$ denote the heart of $\Coh_H(X)$ under the usual t-structure.

We apply Proposition \ref{Prop:Hodge_gr_equiv} to the setting of $\Hcal_\theta$ and $\CS_\theta$. On the one hand, we have a functor $\wt{\gr}_\Hcal: \Hcal_\theta \to\Coh_{G\times \C^\times}(\wt{\Zcal}^R)$. On the other hand, let $\Coh_{G\times\C^\times,\Nilp}(\Cfrak)$ denote the full subcategory of $\Coh_{G\times\C^\times}(\Cfrak)$ whose objects have set-theoretic support contained in $G\times \Nilp$. For every $(N,F)\in \CS_\theta^\heartsuit$, it is well-known (c.f \cite{ginzburg_admissible}) that $\gr^F(N)$ has set-theoretic support in $G\times \mathcal{N}$. Therefore, we also have a functor $\gr_\CS:\CS_\theta\to \Coh_{\Nilp,G\times\C^\times}(\Cfrak)$.


Moreover, equivariant descent along the $T$-bundle $\wt\Zcal^R\to\Zcal^R$ gives a {  t-exact} equivalence
\[
    \Desc_T:\Coh_{G\times T\times\C^\times}(\wt\Zcal^R) \xrightarrow{\sim}\Coh_{G\times\C^\times}(\Zcal^R).
\]

For every $M\in \Hcal_\theta$, $\wt{\gr}_\Hcal(M)$ is $\wt{T}$-equivariant. Let $\wt{\lambda}\in \Lie(\wt{T})^*$ be the integral lift of $\lambda$. Then the kernel of the finite cover $\mu:\wt{T}\to T$ acts on $\wt{\gr}_{\Hcal}(M)$ by $\wt{\lambda}$. Thus $\wt{\gr}_{\Hcal}(M)\otimes \C_{-\tilde{\lambda}}$ is $T$-equivariant. We define
\[
    \gr_{\Hcal,\lambda}(M):=\Desc_T(\wt{\gr}_{\Hcal}(M)\otimes \C_{-\tilde{\lambda}}).
\]
Write
\begin{equation}
    \label{eq:Desc_shorthand} D_\lambda(F):=\Desc_T(F\otimes \C_{-\wt{\lambda}}).
\end{equation}

\subsection{The coherent character functor $\wt{\CH}$}
\label{SSS:affine_CH}
Let \[
    \Yfrak:=\{(g,h,x)\in G\times\Bcal\times\gfrak\mid g^{-1}\cdot x\in h\cdot\nfrak\}.
\] Take $\wt{\Xfrak}$ to be the derived zero fiber of the map
$\Yfrak \to \gfrak, \quad (g,h,x)\mapsto g\cdot x-x.$ Define the map $\rho:\wt{\Xfrak}\to\wt{\Zcal}^R$  induced by the map:
\[
    \rho': \Yfrak\to  T^*Y, \quad (g,h,x)\mapsto (a(g,h),g^{-1}\cdot x,x).
\]
Let $\varpi:\wt{\Xfrak}\to\Cfrak$, $(g,h,x)\mapsto (g,x)$. Let $q:\wt{\Zcal}^R\to\Zcal^R$ be the $T$-torsor, and set
\[
    r:=q\circ\rho:\wt{\Xfrak}\to\Zcal^R,\qquad \varpi:\wt{\Xfrak}\to\Cfrak.
\]
For a map $V\to\Bcal$ from a derived scheme $V$, let $\omega^V_\Bcal$ denote the pullback of the canonical bundle of $\Bcal$ to $V$. Write $\AH:=\Coh_{G\times\C^\times}(\Zcal^R)$ and $b:=\dim(\Bcal)$.
\begin{defn}
\label{defn: CH}
Define the functor
\[
    \wt{\CH}(-):=\varpi_*\bigl(r^*(-)\otimes\omega^{\wt{\Xfrak}}_\Bcal\langle-b\rangle\bigr):\ \AH \longrightarrow\Coh_{\Nilp,G\times\C^\times}(\Cfrak).
\]
\end{defn}
The theorem below justifies the target of $\wt{\CH}$.
\begin{thm}
\label{thm: grCH}
The functor $\wt{\CH}$ maps to $\Coh_{\Nilp,G\times\C^\times}(\Cfrak)$. Moreover, the diagram
\begin{align}
    \label{diagram_grCH} \xymatrix{ \Hcal_\theta \ar[d]^{\gr_{\Hcal,\lambda}}\ar[rrr]^{\mathrm{CH}} &&&\mathcal{CS}_\theta \ar[d]^{\gr_{\mathcal{CS}}}\\ \mathcal{AH} \ar[rrr]^-{\wt{\CH}} &&&\Coh_{\Nilp,G\times\C^\times}(\Cfrak) }
\end{align}
commutes.
\end{thm}
\begin{rem}
Note that for an integral weight $\mu$, $\gr_{\Hcal,\lambda+\mu}(-)\cong \gr_{\Hcal,\lambda}(-)\otimes \Ocal(-\mu,\mu)$ while $r^*(-\mu,\mu)$ is trivial. This explains why $\gr_\CS\circ \CH$ does not depend on the choice of $\lambda$.
\end{rem}
\begin{proof}
Put $X:=G\times\Bcal$, with the $G$-action $k\cdot(g,h)=(kgk^{-1},kh)$, and let $Z_X^R$ be the derived zero fiber of the moment map $\mu_X:T^*X\to\gfrak$.

Combine the $G$-equivariant Lagrangian correspondences \eqref{eq:lagrangian_stack} for $a$ and $p$, and further impose the $\mu_T=0$ condition for $T^*Y$. { Since $\wt{\Zcal}^R$ is the derived zero fiber of the moment map for the $G\times T$-action on $T^*Y$, the following diagram commutes}.
\begin{align}
    \label{diagram_lagrangian_horocycle} \xymatrix{ \wt{\Zcal}^R&\wt{\Zcal}^R\times_Y X\ar[l]_{\varpi_a}\ar[r]^{\rho_a}&Z_X^R\\ &\wt{\Xfrak}\ar[u]\ar[r]\ar[lu]^\rho\ar[rd]_\varpi &\Cfrak\times_G X\ar[d]_{\varpi_p}\ar[u]^{\rho_p}\\ &&\Cfrak }
\end{align}
We first show that the square in diagram \eqref{diagram_lagrangian_horocycle} is derived Cartesian.

Consider
\[
    \Xfrak':=(T^*Y\times_Y X)\times^R_{T^*X}(T^*G\times_G X).
\]
We claim that it is concentrated in a single cohomological degree and
\begin{align}
    \label{eq: X tilde} \Xfrak'\cong\{(g,\bfrak',x)\in G\times\Bcal\times\gfrak\mid g^{-1}\cdot x\in\bfrak'\}.
\end{align}
By $G$-equivariance it suffices to compute the fiber {of the projection $\mathfrak{X}'\rightarrow X$} at $(g,\bfrak)\in X$. Right-translation trivializations give $T^*_{(g,\bfrak)}X\cong\gfrak\times\nfrak$ and
\[
    T_{a(g,\bfrak)}Y\cong \bigl(\gfrak/\nfrak\times\gfrak/(g\cdot\nfrak)\bigr)/\tfrak_g, \qquad\tfrak_g:=\{(\tau,g\cdot\tau)\mid\tau\in\tfrak\}.
\]
For a lift $\hat\eta\in\gfrak$ of $\eta\in\gfrak/\bfrak$, the differentials are
\begin{align*}
    da:\gfrak\times\gfrak/\bfrak &\longrightarrow\bigl(\gfrak/\nfrak\times\gfrak/(g\cdot\nfrak)\bigr)/\tfrak_g, &(\xi,\eta)&\longmapsto\ol{(\hat\eta,\xi+g\cdot\hat\eta)},\\ dp:\gfrak\times\gfrak/\bfrak &\longrightarrow\gfrak, &(\xi,\eta)&\longmapsto\xi.
\end{align*}
A direct check gives
\[
    \Ker(da)=(g\cdot\nfrak)\times\{0\},\qquad \Ker(dp)=\{0\}\times\gfrak/\bfrak.
\]
The images of the two codifferentials are the annihilators of these kernels, and hence
\begin{align}
    \label{iso: X fiber} \Xfrak'_{(g,\bfrak)} =\im(d^*a)\cap\im(d^*p) =(\Ker(da)+\Ker(dp))^\perp =(g\cdot\bfrak)\times\{0\}.
\end{align}
Moreover, $\Ker(da)\cap\Ker(dp)=\{0\}$, so these subbundles intersect transversely. Their derived intersection is therefore classical, proving \eqref{eq: X tilde}. Under \eqref{eq: X tilde} the projections to $T^*Y$ and $T^*G$ are $\rho'$ and $(g,h,x)\mapsto(g,x)$. The compatibility of the co-differentials with the moment maps and derived base change identify the fiber product $(\wt{\Zcal}^R\times_Y X)\times^R_{Z_X^R}(\Cfrak\times_G X)$ with the derived locus in $\Xfrak'$ cut out by
\[
    g\cdot x-x=0,\qquad g^{-1}\cdot x=0\quad\text{in }h\cdot\bfrak/h\cdot\nfrak.
\]
The second equation cuts out $\Yfrak$. Thus $\wt{\Xfrak}$ is the derived zero fiber of the map $\Yfrak \to \gfrak$, $(g,h,x)\mapsto g\cdot x-x$ as desired.

The map from smooth variety $\Yfrak$ to $T^*\Bcal^2$ sends $(g,h,x)$ to $(g^{-1}\cdot x,x)$ over $(h,gh)$. Its moment map is $g^{-1}\cdot x-x$. Consequently, $r$ is the derived base change of this map along $\Zcal^R\to T^*\Bcal^2$ and has finite Tor-dimension. Also, $t_0\wt{\Xfrak}$ is the subvariety
\[
    \{(g,h,x)\mid g\cdot x=x,\ x\in h\cdot\nfrak\}\subset t_0\Cfrak\times\Bcal.
\]
Thus $s$ is proper and its image has set-theoretic support in $G\times\Nilp$. Therefore, the functor $\wt{\CH}$ maps to $\Coh_{\Nilp,G\times\C^\times}(\Cfrak)$.

Applying Proposition \ref{prop:gr_pushforward}, Lemma \ref{lemma_gr_pullback}, and derived base change to \eqref{diagram_lagrangian_horocycle}, we obtain a natural isomorphism:
\begin{align}
    \label{eq: gr CH} \gr_\CS(\CH(-)) \cong \varpi_*\bigl( \rho^*(\wt{\gr}_\Hcal(-)) \otimes\omega^{\wt{\Xfrak}}_\Bcal\langle-b\rangle \bigr).
\end{align}
By the definition of $\gr_{\Hcal,\lambda}$ and equivariant descent along $q$, we have
\[
    q^*\gr_{\Hcal,\lambda}(-) \cong\wt{\gr}_\Hcal(-)\otimes\C_{-\wt{\lambda}}.
\]
Since $r=q\circ\rho$, forgetting the $T$-equivariance, \eqref{eq: gr CH} becomes
\[
    \gr_\CS(\CH(M)) \cong \varpi_*\bigl( r^*(\gr_{\Hcal,\lambda}(M)) \otimes\omega^{\wt{\Xfrak}}_\Bcal\langle-b\rangle \bigr) =\wt{\CH}(\gr_{\Hcal,\lambda}(M))
\]
as required.
\end{proof}



\section{Associated graded of (co)standard objects}\label{SS_assoc_graded}
We compute the images of standard and costandard objects in $\Hcal_\theta$ introduced in \eqref{defn:nabla} and \eqref{defn:Delta} under the map $\gr_{\Hcal,\lambda}$ from diagram \eqref{diagram_grCH}. Note that $\nabla_{w,\theta}^e$ and $\Delta_{w,\theta}^e$ are obtained by taking the $\dagger$ and $!$ extension to the Bott-Samuelson resolution $\wt{BS}_w$ of the integrable connection \eqref{defn: loc sys L w theta}, and then pushing forward to to $\wt{\Bcal}$. Given Proposition \ref{prop:gr_pushforward}, one may carry out the computation on $\wt{BS}_w$ where the extension is across a simple normal crossing (SNC) divisor. In this setting, Saito's theory describes the Hodge filtration explicitly in terms of Deligne's canonical lattice.

{ Here is our motivation for this computation. As we mentioned in the introduction the main computation of this paper is that of $\widetilde{\operatorname{CH}}(\omega^a)$, a $\operatorname{GL}_r\times \C^\times$-equivariant coherent sheaf on the commuting stack $\mathfrak{C}_r$.  The direct computation seems to be very hard if at all possible. Instead we use Theorem \ref{thm: grCH}. It turns out that the object $\omega^a$ (we will work with a version of such objects for $\operatorname{SL}_r$) arises as the Hodge associated graded of $\nabla_{w,\theta}$ -- that in the case of interest coincides with $\Delta_{w,\theta}$ -- for suitable $w,\theta$. The main results of these section, Theorem \ref{thm: gr general theta} and its special case, Corollary \ref{coro: gr nabla_theta} explain how to compute Hodge associated graded sheaves of these objects that, as we will see later, include $\omega^a$. }

\subsection{Steinberg variety}
Let $\Zcal_{\gfrak}:=\wt{\gfrak}\times^R_{\gfrak}\wt{\gfrak}$, which coincides with the nonderived fiber product. Take $\Zcal:=t_0(\Zcal^{R}) =\wt{\Ncal}\times_{\Ncal}\wt{\Ncal}.$

In the regular semisimple part $\wt{\gfrak}^{rs}\times_{\gfrak^{rs}}\wt{\gfrak}^{rs}$, let $\Lambda_{\gfrak,w}$ be the graph of $w$ such that its image in $\Bcal^2$ is the orbit of $(B,wB)$. We write
\[
    \Lambda_w:=\ol{\Lambda_{\gfrak,w}}\cap\Zcal
\]
for the scheme-theoretic {(classical), not derived} intersection inside $\Zcal_\gfrak$.

Let $\nfrak=[\bfrak,\bfrak]$. Define
\[
    \Zcal^e:=\{(h,x)\in\Bcal\times\nfrak\mid x\in h\cdot \nfrak\}.
\]
The $G$-action on $\Zcal$ induces an isomorphism $f:G\times^B\Zcal^e\xrightarrow{\sim}\Zcal$, and set-theoretically
\[
    \Zcal^e=\bigsqcup_{w\in W}T^*_{\Bcal_w}\Bcal.
\]

Similar to $\Zcal^e$, take \[
    \Zcal_\gfrak^e:=\{(\bfrak',x)\in\Bcal\times \bfrak,x\in \bfrak'\}.
\] There is {an isomorphism} $\Zcal_\gfrak\cong G\times^B \Zcal_\gfrak^e.$ The Borel subalgebras that contain a given regular semisimple element are in bijection with $W$. Therefore, writing $\Lambda_{\gfrak,w}^{e }=\{(\bfrak',x)\in\Bcal_{w}\times \bfrak^{rs},x\in \bfrak'\}$, we have
\[
    \{(\bfrak',x)\in\Bcal\times \bfrak^{rs},x\in \bfrak'\}=\sqcup_w \Lambda_{\gfrak,w}^{e }.
\]
Hence we may rewrite the graph $\Lambda_{\gfrak,w}\subset \Zcal_\gfrak$ as $\Lambda_{\gfrak,w}= G\times^B \Lambda_{\gfrak,w}^{e }$. As a result,

\[
    \Lambda_w\cong G\times^B \Lambda_w^e, \text{ where }\Lambda_w^e:=\left(\ol{\Lambda_{\gfrak,w}^{e }}\cap \Zcal^e\right).
\]


Let $\Ocal_X(\alpha)$ be the pullback of the line bundle on $\Bcal$ corresponding to the weight $\alpha$, i.e., $\ind_B^G \C_{-\alpha}$.
Under this convention, dominant weights induce ample line bundles on $\Bcal$. 
For $\mu,\nu\in X^*(T)$, there is a line bundle $\Ocal_{\Zcal}(\mu,\nu)$ on $\Zcal$ restricted from $\Ocal_{\wt{\Ncal}}(\mu)\boxtimes\Ocal_{\wt{\Ncal}}(\nu)$. Let $\rho$ be the half sum of positive roots. We note that even when $\rho$ is not in $X^*(T)$, $\Ocal_\Zcal(\rho,\rho)$ is still a {$G$-equivariant line bundle on $\Zcal$}. 
\subsection{Bott-Samelson resolution}\label{SSS:BS}
Throughout this subsection $w\in W$ is fixed together with a reduced expression $w=s_{k_1}\cdots s_{k_\ell}$. For a simple reflection $s$, write $P_s$ for the associated minimal parabolic. Let
\[
    \wt{BS}_w:=P_{s_{k_1}}\times^B \cdots \times^B P_{s_{k_\ell}}/U,\qquad BS_w:=P_{s_{k_1}}\times^B \cdots \times^B P_{s_{k_\ell}}/B.
\]
Let $\wt p:\wt{BS}_w\to\wt{\Bcal}$ and $p:BS_w\to\Bcal$ be the multiplication maps. We identify $BS_w$ with the variety of sequences $(\bfrak_0,\dots,\bfrak_\ell)$ of Borel subalgebras with $\bfrak_0=\bfrak$ and $\bfrak_{j-1},\bfrak_j$ in relative position $\leq s_{k_j}$, and set
\begin{equation}
    \label{eq:E_D_definition} E_j:=\{\bfrak_{j-1}=\bfrak_j\}\subset BS_w,\qquad D_j:=E_j\times_\Bcal\wt{\Bcal}\subset\wt{BS}_w.
\end{equation}
Write $r_w:\wt{BS}_w\to BS_w$ for the $T$-torsor, canonically identified with the pullback of $\wt\Bcal\to\Bcal$:
\[
    \wt{BS}_w\ \xrightarrow{\ \sim\ }\ BS_w\times_{\Bcal}\wt{\Bcal},\qquad [p_1,\dots,p_\ell]_U\longmapsto \bigl([p_1,\dots,p_\ell]_B,\ p_1\cdots p_\ell U\bigr).
\]
The divisors $E_1,\dots,E_\ell$ are smooth and meet transversally; their complement is isomorphic to the open cell ${\Bcal}_w$. Writing $E:=\sum_jE_j$, we also have that
\[
    D:=\wt{BS}_w\setminus \wt{\Bcal}_w=\bigcup_{j=1}^\ell D_j=r_w^*E,
\]
which is an SNC divisor.

{ Note that $\wt{\Bcal}_w$ is a $T$-bundle over the affine space $\Bcal_w$ so $\pi_1(\wt{\Bcal}_w)\xrightarrow{\sim} \pi_1(T)$. On the other hand, $\pi_1(T)$ is naturally identified with the cocharacter lattice $X_*(T)$}. For $1\leq j\leq \ell$, put
\begin{align}
    \label{def: gamma} \gamma_j:=s_{k_\ell}s_{k_{\ell-1}}\cdots s_{k_{j+1}}(\alpha^\vee_{k_j}).
\end{align}
\begin{lem}\label{lem: BS monodromy}
A positive meridian around $D_j$ represents the class $\gamma_j$ in $\pi_1(\wt{\Bcal}_w)$. By a positive meridian, we mean a counterclockwise loop in a transverse complex disk at a smooth point of $D_j$ based in the complement.
\end{lem}
\begin{proof}
Let $\beta_j\in X_*(T)\cong \pi_1(\wt{\Bcal}_w)$ be the class of a positive meridian about $D_j$. We first show that {$\beta_j=\gamma_j$} in the base case when $w=s_{j_k}$ that  $\beta_j=\alpha_{k_j}^\vee$.

Let $L$ be the Levi subgroup of $P_{s_{j_k}}$. Then $P_{s_{j_k}}/U\cong L/L\cap U$. Therefore it suffices to study the case when $G$ is of semisimple rank $1$. {Such a group is a direct product of a torus with one of $\SL_2$, $\mathrm{PGL}_2$ or $\GL_2$. So it is enough to check that claim for the groups $\SL_2$, $\mathrm{PGL}_2$ or $\GL_2$}. We know that $G/U$ is
 \begin{itemize}
     \item the $\C^\times$-bundle on $\Pbb^1$ associated to $\Ocal(-1)$ for $\SL_2$.
     \item the $\C^\times$-bundle on $\Pbb^1$ associated to $\Ocal(-2)$ for $\mathrm{PGL}_2$.
     \item the $(\C^\times)^2$-bundle on $\Pbb^1$ associated to $\Ocal(-1)\oplus \Ocal(1)$ for $\GL_2$.
 \end{itemize}
Let $s=(12)$ be the simple reflection. According to Lemma \ref{lem: P1 meridian} below, a positive meridian around the complement of $\wt{\Bcal}_s$ corresponds to $1$, $2$ or $(1,-1)$ in $X_*(T)$ respectively, which is exactly the coroot $\alpha_s^\vee$ in each case. Therefore the base case is proved.

In the iterated Bott--Samelson construction, the factors to its right transport this cocharacter by $s_{k_\ell}\cdots s_{k_{j+1}}$. Therefore
	\[
    \beta_j
    =s_{k_\ell}\cdots s_{k_{j+1}}(\alpha_{k_j}^\vee)
    =\gamma_j.
\]
\end{proof}
Take $k\in \Z$. Let $X$ be the total space of $\Ocal_{\Pbb^1}(k)$ minus the zero section and the fiber at $\infty$. Then the isomorphism $X\cong \C\times \C^\times$ induces the identification $\pi_1(X)\cong \pi_1(\C^\times)\cong \Z$.
\begin{lem}\label{lem: P1 meridian}
A positive meridian around the fiber at $\infty$ in $X$ corresponds to $-k\in \pi_1(X)$.
\end{lem}
\begin{proof}
Let $(z,t_0)$ and $(w,t_\infty)$ be the coordinates of the affine chart of Tot$(\Ocal_{\Pbb^1}(k))$ at $0$ and $\infty$ respectively. Then $w=1/z$ and $t_\infty=w^k t_0$. Take a positive meridian $w(\theta)=\epsilon e^{i\theta}$ around the fiber at $\infty$ for $0<\epsilon \ll 1$ and $t_\infty\equiv s$ for some constant $s$. Then in the chart around $0$, $t_0=\frac{s}{\epsilon^k} e^{-ki\theta}$, which represents the class $-k\in \pi_1(X)$.
\end{proof}
\subsection{Deligne lattices}
Let $\Lbf_{w,\theta}$, resp. $\Lbf^!_{w,\theta}$, denote the $\dagger$-, resp. $!$-extension of $\Lbb_{w,\theta}$ from $\wt{\Bcal}_w$ to $\wt{BS}_w$, and let $\Lbf^{\geq-1}_{w,\theta}\subset\Lbf_{w,\theta}$, resp. $\Lbf^{!>-1}_{w,\theta}\subset\Lbf^!_{w,\theta}$, be Deligne's canonical lattices \cite[(3.10)]{saito}, i.e. the lattices whose residue eigenvalues along each component of $D$ lie in $[-1,0)$, resp. $(-1,0]$.

Fix a rational logarithm $\lambda\in X^*(T)\otimes\Q$ of $\theta$ and set
\begin{equation}
    \label{eq:bunch_of_quantities} c_j:=\langle\lambda,\gamma_j\rangle, \qquad n_j:=\lfloor c_j\rfloor+1, \qquad m_j:=\lceil c_j\rceil, \qquad f_j:=c_j-n_j, \qquad e_j:=c_j-m_j.
\end{equation}
Thus $f_j\in[-1,0)$ and $e_j\in(-1,0]$. On $BS_w$, set
\[
    P_\lambda:=\sum_{j=1}^\ell\lfloor c_j\rfloor E_j, \qquad E_\theta:=\sum_{j:\,\theta(\gamma_j)=1}E_j,
\]
{(where $E_j$'s are defined by (\ref{eq:E_D_definition}))} and write $F_\lambda:=r_w^*P_\lambda$ and $D_\theta:=r_w^*E_\theta$.
\begin{lem}\label{lem: lattice classes}
We have isomorphisms of line bundles
	\begin{align}
    \label{eq: lattice classes} \Lbf^{\geq-1}_{w,\theta}&\cong\Ocal(D+F_\lambda), &\Lbf^{!>-1}_{w,\theta}&\cong\Ocal(D+F_\lambda-D_\theta) \cong\Lbf^{\geq-1}_{w,\theta}(-D_\theta).
\end{align}
\end{lem}

\begin{proof}
Let $\mathcal M_{w,\theta} := \Lbf_{w,\theta}(*D) \cong \Lbf^!_{w,\theta}(*D)$ be the common meromorphic extension of $\Lbb_{w,\theta}$ (defined by (\ref{defn: loc sys L w theta})) to $\wt{BS}_w$. Here $\Lbf_{w,\theta}(*D)$, resp. $\Lbf^!_{w,\theta}(*D)$, denotes the sheaf of rational sections that are regular away from $D$ but with possible poles along $D$. Consider the product presentation $U_w\times T \xrightarrow{\sim} \widetilde{\mathcal B}_w, (u,t)\mapsto uwtU$ and let $t:\widetilde{\mathcal B}_w\longrightarrow T$ denote the projection.

Write
\[
    \Lbb_{w,\theta} \cong \mathcal O_{U_w}\boxtimes\mathcal K_\lambda, \text{  {where} } \mathcal K_\lambda := \mathcal D_T\Big/ \sum_{\xi\in\mathfrak t} \mathcal D_T \bigl(\partial_\xi-\langle\lambda,\xi\rangle\bigr).
\]
Let $\psi_0$ be the section {  of $\Lbb_{w,\theta}$} corresponding to $1\boxtimes 1$. It determines a unique rational section $\psi$ of $\mathcal M_{w,\theta}$ such that $\mathcal M_{w,\theta}=\mathcal O_{\widetilde{BS}_w}(*D)\psi.$

Let $V\subset\widetilde{BS}_w$ be a SNC chart with $V\cong\mathbb A^{\ell(w)}\times T$. Let $J_V$ index the boundary components meeting $V$. Let $q_{V,j}$ be the defining equations of $D_j\cap V$ pulled back from $BS_w\cap r_w(V)$, invariant under the $T$-action.

By Lemma \ref{lem: BS monodromy}, we have
	\[
    \operatorname{ord}_{D_j}\chi(t)
    = \langle\chi,\gamma_j\rangle \qquad \text{for }\chi\in X^*(T).
\]
Let $\tau_V: V\to T$ be the projection. Then there is a regular map $a_V:V\to T$ such that, on $V\cap\widetilde{\mathcal B}_w$,
	\begin{align}
    t
    = a_V \left( \prod_{j\in J_V}\gamma_j(q_{V,j}) \right)\tau_V. \label{eq: local-T-coordinate-with-unit}
\end{align}

For $j\in J_V$, put
	\[
    \eta_{V,j} := (d\log a_V) \left(q_{V,j}\partial_{q_{V,j}}\right) \in \mathcal O(V)\otimes\mathfrak t.
\]
Applying $q_{V,j}\partial_{q_{V,j}}$ to \eqref{eq: local-T-coordinate-with-unit} and using the defining equations of $\psi$, we obtain
	\begin{align}
    q_{V,j}\partial_{q_{V,j}}\psi
    = \left( c_j+\langle\lambda,\eta_{V,j}\rangle \right)\psi. \label{eq: local-residue-with-regular-term}
\end{align}
Since $\eta_{V,j}\in q_{V,j}\mathcal O(V)\otimes\mathfrak t$, {  the function} $\langle\lambda,\eta_{V,j}\rangle$ vanishes on $D_j$, we have $\operatorname{Res}_{D_j}(\psi)=c_j.$

Now set
	\[
    \phi_V := \left( \prod_{j\in J_V}q_{V,j}^{-n_j} \right)\psi, \qquad \phi_V^! := \left( \prod_{j\in J_V}q_{V,j}^{-m_j} \right)\psi.
\]
Their residues along $D_j$ are, respectively,
	\[
    c_j-n_j=f_j\in[-1,0), \qquad c_j-m_j=e_j\in(-1,0].
\]
It follows that
	\begin{align}
    \label{eq: lattice generation} \Lbf^{\geq-1}_{w,\theta}(V)=\Ocal(V) \phi_V,\qquad \Lbf^{!>-1}_{w,\theta}(V)=\Ocal(V) \phi_V^!
\end{align}
which are understood via injective maps into $\Mcal_{w,\theta}$. At the intersection of two charts $V$ and $V'$, $\phi_V$ and $\phi_{V'}$ differ by the unit
\[
    \alpha_{V,V'}:=\prod_{j\in J_V}q_{V,j}^{n_j}\prod_{j\in J_{V'}}q_{V',j}^{-n_j}
\]
Since $(\alpha_{V,V'})$ are the transition functions of the line bundle $\mathcal O(\sum_jn_jD_j)$, we obtain that
	\[
    \Lbf^{\geq-1}_{w,\theta} \cong \mathcal O_{\widetilde{BS}_w} \left(\sum_jn_jD_j\right).
\]
Similarly, $\Lbf^{!>-1}_{w,\theta} \cong \mathcal O_{\widetilde{BS}_w} \left(\sum_jm_jD_j\right)$.

Finally,
\[
    n_j=\lfloor c_j\rfloor+1, \qquad m_j= \begin{cases} \lfloor c_j\rfloor+1,&c_j\notin\mathbb Z,\\ \lfloor c_j\rfloor,&c_j\in\mathbb Z, \end{cases}
\]
and $c_j\in\mathbb Z$ if and only if $\theta(\gamma_j)=1.$ Therefore
	\[
    \sum_jn_jD_j
    = D+F_\lambda, \qquad \sum_jm_jD_j
    = D+F_\lambda-D_\theta.
\]
Consequently,
	\[
    \Lbf^{\geq-1}_{w,\theta} \cong \mathcal O(D+F_\lambda), \qquad \Lbf^{!>-1}_{w,\theta} \cong \mathcal O(D+F_\lambda-D_\theta) \cong \Lbf^{\geq-1}_{w,\theta}(-D_\theta),
\]
as desired.
\end{proof}
\subsection{gr of (co)standards}\label{SS_gr_costand}
Let $\widetilde{\Sigma}_w\subset T^*\widetilde{BS}_w$ be the closed subscheme of covectors that annihilate all logarithmic vector fields in $T_{\widetilde{BS}_w}(-\log D)$. With respect to the trivialization $V\simeq V_0\times T$, choose boundary coordinates $q_{V,j}$, remaining coordinates $x_{V,a}$ in the $V_0$-direction, and torus coordinates $z_{V,b}$. If $\xi_{V,j}$, $\eta_{V,a}$, and $\zeta_{V,b}$ denote the principal symbols of $\partial_{q_{V,j}}$, $\partial_{x_{V,a}}$, and $z_{V,b}\partial_{z_{V,b}}$, then
\[
    \widetilde{\Sigma}_w\cap T^*V
    = Z\left( (q_{V,j}\xi_{V,j})_{j\in J_V}, (\eta_{V,a})_a, (\zeta_{V,b})_{b=1}^r \right)
\]
where $Z(-)$ on the right hand side indicates vanishing locus of the listed functions. Its underlying set is the union of the conormal bundles to the boundary strata. Let $\Sigma_w\subset T^*BS_w$ be its quotient by $T$, and write $\wt\iota_w$ and $\iota_w$ for the respective closed embeddings.

For the map $p:BS_w\to\Bcal$, put
\begin{equation}
    \label{eq:Awe_defn} A_w^e:=BS_w\times_{p,\Bcal}T^*\Bcal, \qquad T^*BS_w\xleftarrow{\,\rho_p\,}A_w^e \xrightarrow{\,\varpi_p\,}T^*\Bcal,
\end{equation}
where $\rho_p$ is the codifferential and $\varpi_p$ is the projection. Define
\begin{align*}
    (\mathsf K^\nabla_{w,\lambda})^e &:=(\varpi_p)_*\!\left( \rho_p^*\bigl[(\iota_w)_*\Ocal_{\Sigma_w}\bigr] \otimes\Ocal(P_\lambda)\right),\\ (\mathsf K^\Delta_{w,\lambda})^e &:=(\varpi_p)_*\!\left( \rho_p^*\bigl[(\iota_w)_*\Ocal_{\Sigma_w}\bigr] \otimes\Ocal(P_\lambda-E_\theta)\right).
\end{align*}
Here and blow divisor line bundles are pulled back to $A_w^e$. Let $\mathsf K^\nabla_{w,\lambda}:=\ind_B^G (\mathsf K^\nabla_{w,\lambda})^e$ and $\mathsf K^\Delta_{w,\lambda}:=\ind_B^G (\mathsf K^\Delta_{w,\lambda})^e$.

{ Recall the objects $\nabla_{w,\theta},\Delta_{w,\theta}\in \mathcal{H}^\heartsuit_\theta$ defined by \eqref{defn:nabla} and \eqref{defn:Delta}. }

\begin{thm}\label{thm: gr general theta}
Suppose $w\in W_\theta$. Then we have the following isomorphisms of objects in  $\Coh_{G\times\C^\times}(\Zcal^R)$
	\begin{align}
    \gr_{\Hcal,\lambda}(\nabla_{w,\theta}) &\cong \mathsf K^\nabla_{w,\lambda} \otimes\Ocal(w\lambda-\lambda+\rho,\rho)\langle b-\ell(w)\rangle, \label{eq: gr nabla general}\\ \gr_{\Hcal,\lambda}(\Delta_{w,\theta}) &\cong \mathsf K^\Delta_{w,\lambda} \otimes\Ocal(w\lambda-\lambda+\rho,\rho)\langle b-\ell(w)\rangle. \label{eq: gr Delta general}
\end{align}
\end{thm}

\begin{proof}
By Lemma \ref{prop_gr_induction}, it suffices to compute $\gr \nabla_{w,\theta}^e$ and $\gr \Delta_{w,\theta}^e$ as $\Ad(B)$-equivariant coherent sheaves. We first compute $\gr (\Lbf_{w,\theta})$ and $\gr(\Lbf_{w,\theta}^!)$ on $\wt{BS}_w$ and then apply Proposition  \ref{prop:gr_pushforward}.

{\it Quotient description.} We first compute $\gr \nabla_{w,\theta}^e$ and $\gr \Delta_{w,\theta}^e$ without the $\Ad(B)$-equivariance. On the SNC chart $V\simeq V_0\times T$ with basis $q_{V,j}, x_{V,a}, z_{V,b}$ as above, write $z_{V,b}=\chi_b\circ\tau_V$ for a fixed basis $\chi_1,\dots,\chi_r$ of $X^*(T)$. Let $\xi_1,\dots,\xi_r$ be the dual basis of $X_*(T)$, and set $\lambda_b:=\langle\lambda,\xi_b\rangle.$

Using \eqref{eq: local-T-coordinate-with-unit}, set
	\begin{align*}
    A_{V,j}:=\left\langle \lambda, (d\log a_V)(q_{V,j}\partial_{q_{V,j}}) \right\rangle, \qquad A_{V,a}:=\left\langle \lambda, (d\log a_V)(\partial_{x_{V,a}}) \right\rangle.
\end{align*}
The right $T$-equivariance of $t$ and $\tau_V$ implies that $a_V$ is constant along the $T$-direction. Hence $a_V$ contributes no term in the $z_{V,b}$-directions.

With respect to the local generators $\phi_V$ and $\phi_V^!$ of \eqref{eq: lattice generation}, we have
	\begin{align}
    \Lbf_{w,\theta}|_V &\cong \Dcal_V\Big/ \left( \sum_{j\in J_V} \Dcal_V(q_{V,j}\partial_{q_{V,j}}-f_j-A_{V,j}) + \sum_a \Dcal_V(\partial_{x_{V,a}}-A_{V,a}) + \sum_{b=1}^r \Dcal_V(z_{V,b}\partial_{z_{V,b}}-\lambda_b) \right), \label{iso: * as quotient}\\ \Lbf^!_{w,\theta}|_V &\cong \Dcal_V\Big/ \left( \sum_{j\in J_V} \Dcal_V(q_{V,j}\partial_{q_{V,j}}-e_j-A_{V,j}) + \sum_a \Dcal_V(\partial_{x_{V,a}}-A_{V,a}) + \sum_{b=1}^r \Dcal_V(z_{V,b}\partial_{z_{V,b}}-\lambda_b) \right), \label{iso: ! as quotient}
\end{align}
{ where $f_j,e_j$ are defined in (\ref{eq:bunch_of_quantities}).}

Take $\xi_{V,j}$, $\eta_{V,a}$, and $\zeta_{V,b}$ as above. 
The principal symbols of the relations in both presentations { \eqref{iso: * as quotient}, \eqref{iso: ! as quotient}}  are $(q_{V,j}\xi_{V,j})_{j\in J_V},\ (\eta_{V,a})_a,\ (\zeta_{V,b})_{b=1}^r.$ They form the regular sequence defining $\widetilde{\Sigma}_w\cap T^*V$.

By \cite[(3.10)]{saito}, the trivial filtration ($F_k\Lbf_{w,\theta}^{\geq -1}=\Lbf_{w,\theta}^{\geq -1}$ when $k\geq 0$ and $0$ when $k<0$) induces the Hodge filtration:
	\begin{align*}
    F_k \Lbf_{w,\theta}=\sum_{i+j=k} \left(F_i\Dcal_{\wt{BS}_w}\right)\left(F_j\Lbf_{w,\theta}^{\geq -1}\right).
\end{align*}
The analogous formula holds for $\Lbf^!_{w,\theta}$ with $\Lbf_{w,\theta}^{!> -1}$. Lemma \ref{lem: lattice classes} therefore gives after descent, nonequivariantly,
	\begin{align*}
    D_\lambda\gr(\Lbf_{w,\theta}) &\cong(\iota_w)_*\Ocal_{\Sigma_w}(E+P_\lambda),\\ D_\lambda\gr(\Lbf^!_{w,\theta}) &\cong(\iota_w)_*\Ocal_{\Sigma_w} (E+P_\lambda-E_\theta),
\end{align*}
{where the notation $D_\lambda$ is introduced in (\ref{eq:Desc_shorthand}).}

{\it $\Ad(B)$-equivariance.} Now we identify the $\Ad(B)$ equivariant structure of $D_\lambda\gr(\Lbf_{w,\theta})$ and $D_\lambda\gr(\Lbf^!_{w,\theta})$ by computing the $\Ad(T)$-weight of the rational section $\psi$ from the proof of Lemma \ref{lem: lattice classes}. Recall the isomorphism $U_w\times T\xrightarrow{\sim}\widetilde{\mathcal B}_w$, $(u,t)\mapsto uwtU$. Moreover, $\psi|_{\wt{\Bcal}_w}$ is pullback of the section $1\in \mathcal K_\lambda$ such that $\xi1=\langle \lambda, \xi\rangle 1$ for any $\xi\in \tfrak$. Therefore $\psi|_{\wt{\Bcal}_w}$ is invariant under the left, or equivalently adjoint $U$-action.

For $s\in T$ one has
	\[
    s^{-1}(uwt)s=(s^{-1}us)w\bigl((w^{-1}\cdot s^{-1})ts\bigr),
\]
Thus, with the action on sections defined by $(sf)(x)=f(s^{-1}\cdot x)$, we have
	\begin{align*}
    \xi\cdot\psi
    =\partial_{\xi-w^{-1}\xi}\psi
    =\langle\lambda,\xi-w^{-1}\xi\rangle\psi
    =\langle\lambda-w\lambda,\xi\rangle\psi.
\end{align*}

For every SNC chart $V$, recall also the local sections $\phi_{V}=\prod_{j\in J_V}q_{V,j}^{-n_j}\psi$ satisfying \eqref{eq: lattice generation}. Moreover, the coefficient of $E_j$ in $E+P_\lambda$ is $n_j$. Therefore, $B$-equivariantly we have $\Lbf_{w,\theta}^{\geq -1}\cong \Ocal(D+F_\lambda)\otimes \C_{\lambda-w\lambda}$. Consequently,
	\[
    D_\lambda\gr(\Lbf_{w,\theta}) \cong(\iota_w)_*\Ocal_{\Sigma_w}(E+P_\lambda)\otimes\C_{\lambda-w\lambda}.
\]
Similarly,
	\[
    D_\lambda\gr(\Lbf^!_{w,\theta}) \cong(\iota_w)_*\Ocal_{\Sigma_w} (E+P_\lambda-E_\theta)\otimes\C_{\lambda-w\lambda}.
\]

{\it Direct image.} Finally we determine $D_\lambda\gr(\nabla^e_{w,\theta})$ and $D_\lambda\gr(\Delta^e_{w,\theta})$ by applying Proposition \ref{prop:gr_pushforward}. Let $\widetilde j_w:\wt\Bcal_w\hookrightarrow\wt{BS}_w$ and $j_w:\wt\Bcal_w\hookrightarrow\wt\Bcal$ be the open-cell and Bruhat-cell embeddings. The restriction $u:=\wt p|_{\wt\Bcal_w}$ is an isomorphism onto the Bruhat cell and $\wt p\circ\widetilde j_w=j_w\circ u$. Since $\Lbf_{w,\theta}=(\widetilde j_w)_\dagger\Lbb_{w,\theta}$, functoriality gives
	\[
    \wt p_\dagger\Lbf_{w,\theta} \cong(j_w)_\dagger u_\dagger\Lbb_{w,\theta} \cong\nabla^e_{w,\theta}.
\]
Similarly, using the properness of $\wt p$,
	\[
    \wt p_\dagger\Lbf^!_{w,\theta} \cong \wt p_!(\widetilde j_w)_!\Lbb_{w,\theta} \cong(j_w)_!u_!\Lbb_{w,\theta} \cong\Delta^e_{w,\theta}.
\]
After descent, Proposition \ref{prop:gr_pushforward} gives
	\[
    D_\lambda\gr(\nabla^e_{w,\theta}) \cong \gamma_pD_\lambda\gr(\Lbf_{w,\theta}),\qquad D_\lambda\gr(\Delta^e_{w,\theta}) \cong \gamma_pD_\lambda\gr(\Lbf^!_{w,\theta}),
\]
where $\gamma_p(-)=(\varpi_p)_* (\rho_p^*(-)\otimes\omega_{BS_w/\Bcal}\langle b-\ell(w)\rangle)$. Set $\mathcal T_p:= \omega_{BS_w/\mathcal B}\otimes\mathcal O(E).$ Then equivariantly,
	\[
    \mathcal T_p\cong \mathbb C_{-\rho}\otimes p^*\mathcal O_{\mathcal B}(\rho).
\]

Now the projection formula and the definitions of $\mathsf K^\nabla_{w,\lambda}$ and $\mathsf K^\Delta_{w,\lambda}$ give
	\begin{align*}
    D_\lambda\gr(\nabla^e_{w,\theta}) &\cong\bigl(\mathsf K^\nabla_{w,\lambda}\bigr)^e(\rho) \otimes\C_{\lambda-w\lambda-\rho}\langle b-\ell(w)\rangle,\\ D_\lambda\gr(\Delta^e_{w,\theta}) &\cong\bigl(\mathsf K^\Delta_{w,\lambda}\bigr)^e(\rho) \otimes\C_{\lambda-w\lambda-\rho}\langle b-\ell(w)\rangle.
\end{align*}
Finally we apply Lemma \ref{prop_gr_induction} to conclude the theorem.
\end{proof}
\subsection{Specializations}
When $\lambda$ satisfies certain conditions, the statements of Theorem \ref{thm: gr general theta} simplify, as we will see in Corollary \ref{coro: gr nabla_theta}.

For a subset $I\subset\{1,\ldots,\ell\}$, put $E_I:=\sum_{j\in I}E_j$ and define
\begin{align*}
    (\mathsf K_{w,I})^e:= (\varpi_p)_*\left( \rho_p^*\bigl((\iota_w)_*\Ocal_{\Sigma_w}\bigr) \otimes\Ocal(-E_I) \right).
\end{align*}
where $\Sigma_w$ is defined in the beginning of Section \ref{SS_gr_costand}. Write $\mathsf K_{w,I}=\ind_B^G (\mathsf K_{w,I})^e$. Recall (from e.g \cite[1.2]{BR}) that the category $D^b\Coh_{G\times\mathbb{G}_m,\Zcal} (\wt\Ncal\times\wt\Ncal)$ comes equipped with the convolution:
\begin{align}
    \label{eq: BR convolution order} \Fcal\star\Gcal:=p_{13*} \bigl(p_{12}^*\Gcal\otimes p_{23}^*\Fcal\bigr).
\end{align}
Let $\omega$ denote the dualizing sheaf with its natural equivariant structure.
\begin{lem}
\label{lem:BS_conormal}
The following statements hold.
	\begin{enumerate}
\item $\mathsf K_{w,\varnothing}\cong\Ocal_{\Lambda_w}$.
\item $\mathsf K_{w,\{1,\ldots,\ell\}}\cong\omega_{\Lambda_w}\langle \ell(w)-b\rangle$, where $b:=\dim(\Bcal)$.
\item $\mathsf K_{w,I}\cong M_{k_\ell}^{I}\star\cdots\star M_{k_1}^{I}$, where
\[
    M_{k_j}^{I}:= \begin{cases} \omega_{\Lambda_{s_{k_j}}}\langle 1-b\rangle,&j\in I,\\ \mathcal O_{\Lambda_{s_{k_j}}},&j\notin I. \end{cases}
\]
	\end{enumerate}
We regard these sheaves on $\Zcal$ via extension by $0$, and suppress it in the notation.
\end{lem}
\begin{proof}
Put $\Xi_w^e:=A_w^e\times^R_{T^*BS_w}\Sigma_w$, where $A^e_w$ is defined by (\ref{eq:Awe_defn}). Let $q:=\Xi_w^e\to T^*\Bcal$ be the composition of the projection to $A_w^e$ with $\varpi_p$. Then
	\begin{align*}
    (\varpi_p)_* \rho_p^*\bigl((\iota_w)_*\Ocal_{\Sigma_w}\bigr) \cong q_*\Ocal_{\Xi_w^e}.
\end{align*}
Define
\[
    C_w:=\Lambda_{s_{k_1}}\times^R_{\wt\Ncal}\cdots\times^R_{\wt\Ncal}\Lambda_{s_{k_\ell}}.
\]
The flags in $C_w$ are ordered $(\bfrak_0,\ldots,\bfrak_\ell)$. Let $C_w^{e}$ denote the fiber at $\bfrak_0=\bfrak$.

We first show there is a $B\times\C^\times$-equivariant isomorphism of derived schemes
	\begin{align}
    \Xi_w^e \xrightarrow{\sim}C_w^e. \label{eq:BS_conormal_comparison}
\end{align}
Let $\pi :T^*BS_w\longrightarrow BS_w$ be the projection. The tautological one-form on $T^*BS_w$, followed by the natural inclusion $\Omega^1_{BS_w}\hookrightarrow\Omega^1_{BS_w}(\log E)$, defines a section
	\begin{align*}
    \tau_{\log}\in \Gamma\left( T^*BS_w,\pi^*\Omega^1_{BS_w}(\log E) \right).
\end{align*}
On an SNC chart with coordinates $t_1,\ldots,t_r,y_1,\ldots,y_m$ in which $E$ is defined by $t_1\cdots t_r=0$ and dual cotangent coordinates $\eta_1,\ldots,\eta_r,\zeta_1,\ldots,\zeta_m$, the coefficients of $ \tau_{\log}$ in the logarithmic basis are $t_1\eta_1,\ldots,t_r\eta_r,\zeta_1,\ldots,\zeta_m.$ Hence, $\Sigma_w=Z(\tau_{\log})${ , the classical vanishing locus which in this case also coincides with the derived vanishing locus.}

Pulling $\tau_{\log}$ back along $\rho_p:A_w^e\to T^*BS_w$ gives the logarithmic one-form
	\begin{align*}
    s_w:=\rho_p^*\tau_{\log} \in \Gamma\left( A_w^e,\pi_A^*\Omega^1_{BS_w}(\log E) \right),
\end{align*}
where $\pi_A:A_w^e\to BS_w$ is the projection. Derived zero loci (denoted by $Z^R(-)$) commute with base change, so
	\begin{align}
    Z^R(s_w) \cong \Xi_w^e. \label{eq:BS-derived-zero-locus}
\end{align}

To show \eqref{eq:BS_conormal_comparison}, we shall prove that $Z^R(s_w)\cong C_w^e$. We do so by induction along the Bott-Samelson tower. For $0\leq j\leq\ell$, put
	\begin{align*}
    BS^{(j)}:= P_{s_{k_1}}\times^B\cdots\times^B P_{s_{k_j}}/B, \qquad E^{(j)}:=E_1+\cdots+E_j,
\end{align*}
and let $\pi_j:BS^{(j)}\longrightarrow BS^{(j-1)}$ forget the last flag. Since $\pi_j$ is a smooth $\mathbb P^1$-bundle and $E_j$ is a section, there is an exact sequence
	\begin{align*}
    0\longrightarrow T_{BS^{(j)}/BS^{(j-1)}}(-\log E_j) \longrightarrow T_{BS^{(j)}}(-\log E^{(j)}) \longrightarrow \pi_j^*T_{BS^{(j-1)}}(-\log E^{(j-1)}) \longrightarrow0.
\end{align*}

The first term has rank one. For $0\leq j\leq \ell$, take \[
    A_j=BS^{(j)}\times_\Bcal T^*\Bcal\cong \{(\bfrak_0,\dots,\bfrak_j,x),x\in \nfrak_{\bfrak_j}\}
\]
such that $A_\ell=A_w^e$. Let $a_j:A_j\to BS^{(j)}$ be the projection. Define a similar one-form $s_j\in\Gamma\bigl(A_j,a_j^*\Omega^1_{BS^{(j)}}(\log E^{(j)})\bigr)$. Take an $\mathfrak{sl}_2$-triple $(e,h,f)$ associated with the simple root $\alpha_{k_j}$, with $e\in\mathfrak n_{\mathfrak b_{j-1}}$. Let $u$ be the affine coordinate on the standard chart of $P_{s_{k_j}}/B\simeq\mathbb P^1$ parametrized by $\exp(uf)$ such that $E_j$ is given by $u=0$. Then affine-locally, $\mathfrak b_j(u) =\operatorname{Ad}(\exp(uf))\mathfrak b_{j-1}$. Let $\xi$ be the coordinate along the root space spanned by $e$. For $x\in \nfrak_{\bfrak_j}(u)$, we may write
	\begin{align*}
    x=x'+\xi\operatorname{Ad}(\exp(uf))e=x'+\xi(e-uh-u^2f)
\end{align*}
where $x'\in \nfrak_{P_{s_{k_j}}}$. Then $x \in\nfrak_{\bfrak_{j-1}}$ is equivalent to $u\xi=0$, which is the $j$-th $\Pbb^1$ direction of $s_j$ up to a unit. 
On the zero locus $\Upsilon_j$ {in} $A_j$, defined by $x\in\nfrak_{\bfrak_{j-1}}$, the form $s_j$ descends to the dual of the quotient $\left(\pi_j^*T_{BS^{(j-1)}}(-\log E^{(j-1)}))\right)|_\Upsilon$ to be the pullback of the section $s_{j-1}$ for the truncated word. In all, we have
\[
    Z^R(s_j)\cong \Upsilon_j\times^R_{A_{j-1}}Z^R(s_{j-1}).
\]
Since $\Upsilon_j\cong A_{j-1}\times_{\wt\Ncal}\Lambda_{s_{k_j}} $, induction on length therefore gives
	\begin{align*}
    Z^R(s_w) \cong \Lambda_{s_{k_1}} \times^R_{\wt\Ncal}\cdots \times^R_{\wt\Ncal} \Lambda_{s_{k_\ell}}\big|_{\bfrak_0=\bfrak} (=C_w^e).
\end{align*}
Together with \eqref{eq:BS-derived-zero-locus}, this proves \eqref{eq:BS_conormal_comparison}.

By \cite[Theorem 2.2.1 and Remark 2.2.2(2)]{BR}, $\Ocal_{\Lambda_{s_{k_\ell}}}\star\cdots\star \Ocal_{\Lambda_{s_{k_1}}} \cong \Ocal_{\Lambda_w}.$ Thus part (1) follows.

For part (2), we have
\[
    K_{w,\{1,\dots,\ell\}}=(\varpi_p)_*\left( \rho_p^*\bigl((\iota_w)_*\Ocal_{\Sigma_w}\bigr) \otimes\Ocal(-E)\right)=\ind_B^G q_*\left(\Ocal_{\Xi_w^e}(-E)\right),
\]
which we need to identify with $\omega_{\Lambda_w}\langle \ell(w)-b\rangle$.

Let $\omega^\bullet_{\Sigma_w^e}$ denote the dualizing complex of $\Sigma_w^e$. By \cite[Theorem 2.2.1]{BR}, $\Lambda_w^e$, is Cohen-Macaulay. Therefore we can take duality of $q_*\Ocal_{\Xi_w^e}\cong \Ocal_{\Lambda_w^e}$ from part (1) and obtain
\[
    q_*\omega^\bullet_{\Xi_w^e}\cong \omega_{\Lambda_w^e}[b].
\]
Since $\omega_{\Lambda_w}=\ind_B^G (\omega_{\Lambda_w^e}\otimes \C_{2\rho})$, it suffices to show that
\begin{align}
    \label{iso: dual} \Ocal_{\Xi_w^e}(-E)\cong \omega^\bullet_{\Xi_w^e}\otimes \C_{-2\rho}\langle \ell(w)-b\rangle[-b].
\end{align}

Recall isomorphism \eqref{eq:BS-derived-zero-locus} realizing $\Xi_w^e$ as the vanishing locus $Z^R(s_w)$ inside the rank-$\ell(w)$ vector bundle $\Vcal:=\pi_A^*\Omega^1_{BS_w}(\log\ E)$. The Koszul complex of this section computes the dualizing complex $\omega_{\Xi_w^e}$. Therefore
\begin{align}
    \label{iso: omega} \omega^\bullet_{\Xi_w^e}=&(\omega_{A_w^e}\otimes \det \Vcal\langle-\ell(w)\rangle)|_{\Xi_w^e}[\dim(A_w^e)-\mathrm{rk}(\Vcal)]
\end{align}
Let $p_0,p_\ell: BS_w\to \Bcal$ be the endpoint maps. Then
\[
    \omega_{BS_w}(E)\cong p_0^*\Ocal_\Bcal(-\rho)\otimes p_\ell^*\Ocal_\Bcal(-\rho)\cong \C_{-\rho}\otimes p_\ell^*\Ocal_\Bcal(-\rho).
\]
Then $\omega_{A_w^e}\cong \pi_A^*\left(\omega_{BS_w}\otimes p_\ell^*\omega^{-1}_{\Bcal}\right)\langle b\rangle$ while $\det \Vcal\cong \pi_A^*\omega_{BS_w}(E).$ Plugging these into \eqref{iso: omega}, we obtain
\begin{align*}
    \omega_{\Xi_w^e}^\bullet\cong &\pi_A^*\left(\omega_{BS_w}\otimes p_\ell^*\omega^{-1}_{\Bcal}\otimes \omega_{BS_w}(E)\right)|_{\Xi_w^e}\langle b-\ell(w)\rangle [b]\\ \cong&\pi_A^*\left(\C_{-2\rho}\otimes p_\ell^*\Ocal_\Bcal(-2\rho)(-E)\otimes p_\ell^*\omega_\Bcal^{-1}\right)|_{\Xi_w^e}\langle b-\ell(w)\rangle [b]\\ \cong&\Ocal_{\Xi_w^e}(-E)\otimes \C_{-2\rho}\langle b-\ell(w)\rangle [b],
\end{align*}
which proves \eqref{iso: dual}.

For part (3), similar to (2), we need to show
\begin{align}
    \label{iso: (3)} q_*\mathcal O_{\Xi_w^e}(-E_I)\cong \left(M_{k_\ell}^{I}\star\cdots\star M_{k_1}^{I}\right)^e
\end{align}
Recall the isomorphism  \eqref{eq:BS_conormal_comparison}. Let $p_j: C_w^e\to \Lambda_{s_{k_j}}$ be the map defined by $(\bfrak_0,\dots,\bfrak_\ell,x)\mapsto (\bfrak_{j-1},\bfrak_j,x)$. Then
	\begin{align*}
    \mathcal O_{\Xi_w^e}(-E_I) \cong \bigotimes_{j=1}^{\ell}p_j^*\Lcal_j, \qquad \Lcal_j:= \begin{cases} \Ocal_{\Lambda_{s_{k_j}}}(-E_j),&j\in I,\\ \Ocal_{\Lambda_{s_{k_j}}},&j\notin I. \end{cases}
\end{align*}
On the other hand, $\Lcal_j\cong M_{k_j}^I$. Let $p_{0\ell}$ be the end point projection $C_w\to \wt{\Ncal}\times\wt{\Ncal}$. Then
\[
    M_{k_\ell}^{I}\star\cdots\star M_{k_1}^{I}\cong (p_{0\ell})_*\left(\bigotimes_{j=1}^{\ell}\operatorname{pr}_j^*\Lcal_j\right).
\]

Let $p_{0\ell}^e:C_w^e\to \nfrak\times\wt{\Ncal}$ be the restriction of $p_{0\ell}$ at the fixed Borel and $r: \nfrak\times\wt{\Ncal}\to \wt{\Ncal}$ is the projection. Then the isomorphism \eqref{eq:BS_conormal_comparison} intertwines $q$ and $r\circ p_{0\ell}^e$. Therefore \eqref{iso: (3)} follows.
\end{proof}

\begin{coro}\label{coro: gr nabla_theta}
Suppose $w\in W_\theta$ and $c_j\in[0,1)$ for every $j$. We have the following identifications in $\Coh_{G\times\mathbb{G}_m}(\Zcal^R)$:
	\begin{enumerate}
\item $\gr_{\Hcal,\lambda}(\nabla_{w,\theta})\cong \Ocal_{\Lambda_w}\otimes\Ocal(w\lambda-\lambda+\rho,\rho)\langle b-\ell(w)\rangle.$


\item $\gr_{\Hcal,\lambda}(\Delta_{w,\theta})\cong \left(M_{k_\ell}\star\cdots\star M_{k_1}\right) \otimes\Ocal(w\lambda-\lambda+\rho,\rho)\langle b-\ell(w)\rangle$ where \[
    M_{k_j}=\begin{cases}\omega_{\Lambda_{s_{k_j}}}\langle 1-b\rangle,&\theta(\gamma_j)=1,\\[2pt] \mathsf \Ocal_{\Lambda_{s_{k_j}}},&\theta(\gamma_j)\neq1.\end{cases}
\] 
	\end{enumerate}
\end{coro}
\begin{proof}
The condition $c_j\in[0,1)$ for all $j$ gives $P_\lambda=0$. Hence the corollary follows from Theorem \ref{thm: gr general theta} and Lemma \ref{lem:BS_conormal}.
\end{proof}

\begin{rem}
An alternative proof of this result (for $\theta=1$) will appear in \cite{BIR}, see also \cite[Remark 6.18]{dv22}. Notice that the present argument relies on the result of  \cite{BR} proved by reduction to positive characteristic, while \cite{BIR}, \cite{dv22} employ methods intrinsic to Hodge modules theory.
\end{rem}

\section{Traces of length $0$ elements}\label{S_traces}
In this section, take $G=\SL_n$ with Weyl group $W=S_n$. Let $W^{ex}=X^*(T)\rtimes W$ be the extended affine Weyl group.
\subsection{Length $0$ elements}
For $1\leq a\leq n$, let $\omega_a$ be the $a$-th fundamental weight. Put $\omega_0=\omega_n=0$ and $c\in S_n$ be the Coxeter element $s_1\cdots s_{n-1}$. For any $\nu\in P$, denote by $t_\nu\in W^{ex}$ the translation by $\nu$.

\begin{lem}
The length zero subgroup $\Omega$ of $W^{ex}$ is generated by $t_{\omega_1}c $ with $( t_{\omega_1}c)^a=t_{\omega_a}c^a$ for all $1\leq a\leq n$.
\end{lem}

\begin{proof}
The element $t_{\omega_1}c$ fixes the fundamental affine alcove and hence generates $\Omega$. Moreover, $c^{i}t_{\omega_1}c^{-i}=t_{c^i(\omega_1)}$ and $\prod_{j=0}^{a-1}t_{c^j(\omega_1)}=t_{\omega_a}$, which gives the formula for its $a$-th power.
\end{proof}
Since $c^a\rho/n-\rho/n=-\omega_a$, Corollary \ref{coro: gr nabla_theta}(1), applied with $\lambda=\rho/n$, gives the following special case. Its hypothesis holds for every $w\in W_\theta$: for any positive coroot $\beta^\vee$ one has $\langle\rho/n,\beta^\vee\rangle=\mathrm{ht}(\beta^\vee)/n\in[1/n,(n-1)/n]\subset[0,1)$. The same calculation gives $e^{2\pi i\rho/n}(\beta^\vee)=e^{2\pi i\operatorname{ht}(\beta^\vee)/n}\neq1$, so $D_\theta=0$ for every $w\in W_\theta$.
\begin{coro}\label{coro:gr length 0}
For $\theta=e^{2\pi i\rho/n}$ and $1\leq a\leq n$, $\Delta_{c^a,\theta}\cong IC_{c^a,\theta}\cong\nabla_{c^a,\theta}.$ Moreover,
\[
    \gr_{\Hcal,\lambda}(\nabla_{c^a,e^{2\rho\pi \sqrt{-1}/n}})\otimes \Ocal(-\rho,-\rho)\cong \Ocal_{\Lambda_{c^a}}(-\omega_a,0)\langle b-a(n-a)\rangle.
\]
\end{coro}
{ It is easy to see that under the two realization equivalence from \cite{Bezrukavnikov_Hecke} the right hand side indeed corresponds to the length $0$ object in the affine Hecke category for $\operatorname{PGL}_n$ corresponding to $a$.}

\subsection{Cuspidal $D$-modules}\label{SS:cuspidal}
Let $\Unip$ be the unipotent cone. We use the subscript $_r$ to denote regular locus. The regular unipotent orbit $\Unip_r\subset \Unip$ has fundamental group $\pi_1(\Unip_r)=\Z/n\Z$. For $0\leq a\leq n-1$, let $\Kbb_a$ be the integrable connection on $\Unip_r$ corresponding to $e^{2\pi a\sqrt{-1}/n}\in (\Z/n\Z)^*$. Define the {$D_G$-module} $\Nbb_a$ to be the intermediate extension of $\Kbb_a$ to $G$. Equip $\Kbb_a$ with the type-$(0,0)$ Hodge structure. When $a$ is coprime to $n$, $\Nbb_a$ is clean, i.e, it is also the $\dagger$- or $!$-extension of $\Kbb_a$ \cite[(17.8.4)]{Lusztig_character_sheaves_IV}. For every $z\in Z_G$, let $z\Nbb_a$ denote the pushforward of $\Nbb_a$ under the multiplication map by $z$.

{  The following theorem is one of the main results of this paper.}

\begin{thm}\label{thm: coxeter_cuspidal}
When $a$ is coprime to $n$, there is an isomorphism of Hodge modules
	\[
    \CH(\nabla_{c^a,e^{2\pi\rho\sqrt{-1}/n}})\cong\bigoplus_{z\in Z_G} z\Nbb_{a}(-m_a)
\]
for $m_a=(n-1+a(n-a))/2$.
\end{thm}

Recall that $U_a=U\cap c^a\cdot U_-$. Let $S_a=U_{c^a}c^a$.
\begin{lem}\label{lemma_section}
When $a$ is coprime to $n$, the map $B/Z(G) \times S_a \to Bc^a B$, $(b,x)\mapsto b\cdot x$ is an isomorphism.
\end{lem}
\begin{proof}
To apply \cite[Theorem 1.5]{NTY}, we first check the quasi-convexity condition, which we recall. Let $\Phi^\pm(w):=\{\alpha\in \Phi^\pm, w^m\alpha\in \Phi^\pm\text{ for all }m\in \Z\}$ and $\Phi(w)=\Phi^+(w)\sqcup \Phi^-(w)$. For a positive root $\alpha$, define the sign change time by $n_w(\alpha)=\min\{m\geq 1, w^m\alpha\in \Phi^-\}$. Then { $w$ is called {\it quasi-convex} if} $\Phi(w)$ is empty and \begin{equation}
    \label{eq:inequality} n_w(\alpha+\beta)\leq \max\{n_w(\alpha),n_w(\beta)\}
\end{equation}
for $\alpha,\beta\in \Phi^+$ such that $\alpha+\beta$ is a root. For $w=c^a$ and a positive root {$\alpha_{ij}=\epsilon_i-\epsilon_j$} for $i<j$, as well as $1\leq m<n$, let $q_m$ be an integer in $[1,n-1]$ such that $q_m\equiv -am$ modulo $n.$ Then $w^m\alpha_{ij}$ becomes negative exactly when $i\leq q_m< j$. Since $a$ is coprime to $n$, we see that $\Phi(w)$ is empty. Moreover, for $i<j<k$, the first sign change time of $[i,k)$ is the minimum of those for $[i,j)$ and $[j,k)$ and hence (\ref{eq:inequality}) holds. Therefore, \cite[Theorem 1.5]{NTY} implies that the map
\begin{align}
    \label{iso: U Sa} U\times S_a\xrightarrow{\sim}Uc^aU,\ (u,x)\mapsto u\cdot x
\end{align}is an isomorphism.

Since $c^a$ is an $n$-cycle, the map $\delta: T/Z(G)\to T$, $t\mapsto(c^a)^{-1}tc^a t^{-1}$ is an isomorphism. Moreover, we have the Bruhat coordinate
\[
    U_w\times T\times U\xrightarrow{\sim} BwB,\  (u,h,v)\mapsto uwhv.
\]
As a result, the map
\begin{align}
    \label{iso: T Sa} T/Z(G)\times Uc^aU\xrightarrow{\sim}Bc^aB,\ (t,x)\mapsto t\cdot x
\end{align}
is also an isomorphism with an inverse map
\[
    uc^ahv\mapsto (\ol{t},(\ol{t}^{-1}\cdot u)c^a(\ol{t}^{-1}\cdot v) ,\text{ where } \ol{t}=\delta^{-1}(h).
\]
Composing \eqref{iso: U Sa} and \eqref{iso: T Sa}, we obtain the desired isomorphism. 
\end{proof}
\begin{lem}\label{lem: coxeter_cuspidal}
Let $\theta=e^{2\pi\rho\sqrt{-1}/n}$ and let $a$ be coprime to $n$. Then forgetting the Hodge module structure, $\CH(\nabla_{c^a,\theta})$ is a direct sum of shifts of copies of $\bigoplus_{z\in Z_G}z\Nbb_a$.
\end{lem}

\begin{proof}
{  Recall  the horocycle correspondence
	\begin{equation}
    \label{eq:horocycle} G\xleftarrow{f}G\times^B G\xrightarrow{q} G\times^B\wt\Bcal.
\end{equation}}
Put $w=c^a$ and $K:=\CH(\nabla_{w,\theta})=f_\dagger q^\dagger\nabla_{w,\theta}$. Under the identification
	\[
    G\times^B G\xrightarrow{\sim}G\times\Bcal,\qquad [g,h]\longmapsto(ghg^{-1},gB),
\]
{  the map $f$ becomes  $f([g,h])=ghg^{-1}$.} By Lemma \ref{lemma_section},
	\[
    G\times^B(BwB)\cong G/Z_G\times S_a.
\]

{ Using Corollary \ref{coro:gr length 0} and the smoothness of $q$, we see that} $q^\dagger\nabla_{w,\theta}$ is the $!$-, $\dagger$-, and intermediate extension of an integrable connection $\Fbb_a$ on $G/Z_G\times S_a$. Moreover, $\Fbb_a$ is pulled back from the cover $G\to G/Z_G$ through the character $\chi_a$ of $\mu_n$, $z\mapsto z^a$. Equip $\Fbb_a$ with its weight-zero polarizable Hodge structure. Its intermediate extension is pure, and $f$ is proper. Hence the decomposition theorem expresses $K$ as a direct sum of shifts of simple equivariant pure Hodge modules.

The center acts on every stalk of $K$ through $\chi_a$. For $x\in G$, let $d(x)$ be the gcd of its Jordan-block sizes. Then $\pi_0(C_G(x))\cong\mu_{d(x)}$ and $\mu_n=Z_G\twoheadrightarrow\mu_{d(x)}.$ The stabilizer action on a stalk factors through $\pi_0(C_G(x))$. Therefore a nonzero stalk forces $\chi_a$ to factor through $\mu_{d(x)}$. Since $\chi_a$ has order $n$, this is possible only when $d(x)=n$, i.e when $x\in \cup_{Z_G} z\Unip_r$.

It follows that every simple summand of $K$ is the intermediate extension of an irreducible equivariant integrable connection on one of the orbits $z\Unip_r$. For $x\in z\Unip_r$, $\pi_0(C_G(x))=Z_G=\mu_n.$ Thus the central character determines the integrable connection: after forgetting Hodge structures, it is $z\Kbb_a$, and its intermediate extension is $z\Nbb_a$.

Finally, the $T$-action on $G\times^B\wt{\Bcal}$ restricts to a $Z_G$ action, which lifts to a $Z_G$-action on $G\times^B G$. With respect to these actions and the translation $Z_G$-action on $G$, the maps $q$ and $f$ in \eqref{eq:horocycle} are $Z_G$-equivariant. Hence $\CH$ intertwines these two actions. Since $\nabla_{c^a,\theta}$ is $\theta$-monodromic,  {$z\nabla_{c^a,\theta}\cong \nabla_{c^a,\theta}$}. It follows that $\CH(\nabla_{c^a,\theta})$ is invariant under central translation and every summand $z\mathbb N_a$ occurs with the same shifts and multiplicities.
\end{proof}
\newcommand{\DR}{\mathrm{DR}}
\begin{proof}[Proof of Theorem \ref{thm: coxeter_cuspidal}]
We first show that $\CH(\nabla_{c^a,\theta})$ is concentrated in cohomological degree $0$. Use $\CH=f_\dagger q^\dagger$ from the proof of Lemma \ref{lem: coxeter_cuspidal}. Consider the composition $\kappa: G\times^B(Bc^aB)\xrightarrow{j} G\times^B G \xrightarrow{f} G$. We have the clean integrable connection $\Fbb_a$ on $G\times^B(Bc^aB)$ such that $\CH(\nabla_{c^a,\theta})=\kappa_\dagger \Fbb_a=\kappa_! \Fbb_a$, since $f$ is proper. As $\kappa$ is affine, we therefore have that $\CH(\nabla_{c^a,\theta})\in D^{\leq 0}(G)\cap D^{\geq 0}(G),$ {  where $D^{\leq 0}(G), D^{\geq 0}(G)$ stand for the nonpositive and nonnegative parts of the usual t-structure on the $G$-equivariant derived category of D-modules on $G$.}

Thus Lemma \ref{lem: coxeter_cuspidal} gives
    \begin{align}
    \label{eqn: Qa} \CH(\nabla_{c^a,\theta}) \cong \Bigl(\bigoplus_{z\in Z_G}z\Nbb_a\Bigr)\otimes Q_a
\end{align}
for a finite-dimensional Hodge structure $Q_a$. We determine $Q_a$ from the stalk at a regular unipotent element $x$.

{Thanks to Lemma \ref{lemma_section} $G\times^B (Bc^aB)\xrightarrow{\sim} G/Z_G\times S_a$.} So, the fiber of $\kappa$ at $x$ is
    \[
    X_x:=\{(gZ_G,s)\in G/Z_G\times S_a\mid gsg^{-1}=x\}.
\]
Projection to the second coordinate $\pi_x:X_x\rightarrow S_a\cap\Unip_r$ defines a torsor with fiber being $C_G(x)/Z_G\cong\mathbb A^{n-1}.$

Let $i_x:\{x\}\to G$ be the embedding. Write $\DR$ for the de Rham functor and let $L$ be the rank-$1$ local system underlying $\Fbb_a$. Proper base change gives
    \begin{align}
    \label{eqn: fiber CH nabla} i_x^*\DR\CH(\nabla_{c^a,\theta}) \cong R\Gamma_c(X_x,L|_{X_x})[\dim(G\times^B(Bc^aB))],
\end{align}
where $i_x^*$ denotes the constructible inverse image. Given \eqref{eqn: Qa}, the left hand side of \eqref{eqn: fiber CH nabla} can be identified with $Q_a[\dim \Unip_r]$

Let $\chi_c$ denote the compactly-supported Euler characteristic. Then by \cite[Proposition 2.5.4(ii)]{Dimca}, $\chi_c(X_x,L|_{X_x})=\chi_c(X_x,\C)$. Moreover, by \cite[Theorem 3.2]{kim}, $\chi_c(S_a\cap\Unip_r,\C_{S_a\cap\Unip_r})=1$.  Thus $\dim(Q_a)=1$. Now the weight of $\Nbb_a$ is $\dim(\Unip_r)$ and the weight of $j_\dagger \Fbb_a$ is $\dim(G\times^B(Bc^aB))$. Therefore the weight of $Q_a$ is
\[
    \dim(G\times^B(Bc^aB)-\dim(\Unip)=(\dim(G)+a(n-a))-(\dim(G)-(n-1))=2m_a.
\]

For a Hodge module $M$, let $\ol{M}$ denote its complex conjugation, as defined in \cite[13.4.c]{SabbahSchnell}. We shall show that $\ol{Q_a}\cong Q_a$, which would imply that $Q_a\cong \C(-m_a)$. For a VHS, taking conjugation preserves the Hodge type conjugates the monodromy. Moreover, taking conjugation commutes with $!$ and $\dagger$ extensions. Therefore, $\ol{\nabla_{c^a,\theta}}\cong \nabla_{c^a,\theta^{-1}}$, $\ol{\Nbb_a}\cong \Nbb_{n-a}$. Hence
\[
    \CH(\nabla_{c^a,\theta^{-1}}) \cong \Bigl(\bigoplus_{z\in Z_G}z\Nbb_{n-a}\Bigr)\otimes \ol{Q_a}
\]

Let $\sigma:G\to G$ be the group automorphism induced by the Dynkin diagram involution preserving $B$ and $T$: {$g\mapsto J(g^{-1})^TJ^{-1}$ where $J$ is the anti-diagonal of all $1$'s.} Then $\sigma(c)=c^{-1}$ and $\sigma(z)=z^{-1}$ for $z\in Z(G)$. Moreover, $\sigma$ acts on $X^*(T)$ by $-w_0$ where $w_0$ is the longest element. Hence $\sigma$ fixes $\rho$, $\theta$ and $\theta^{-1}$. Consider the anti-automorphism $G\to G$, $g\mapsto \sigma(g)^{-1}$, which induces automorphisms of $Y$ and $G\times \Bcal$. Denote that one of $Y$ by $\wt{\tau}$. Then by base change,
\[
    \CH\circ \wt{\tau}^\dagger\cong \tau^\dagger \circ \CH.
\] Applying this to $ \CH(\nabla_{c^a,\theta^{-1}})$, noting $\wt{\tau}^\dagger(\nabla_{c^a,\theta^{-1}})\cong \nabla_{c^a,\theta}$ and ${\tau}^\dagger(\Nbb_a)\cong \Nbb_{n-a}$, we obtain
\[
    \Bigl(\bigoplus_{z\in Z_G}z\Nbb_a\Bigr)\otimes Q_a\cong \CH\circ \wt{\tau}^\dagger(\nabla_{c^a,\theta^{-1}})\cong \tau^\dagger \circ \CH(\nabla_{c^a,\theta^{-1}})\cong \Bigl(\bigoplus_{z\in Z_G}z\Nbb_{a}\Bigr)\otimes \ol{Q_a}.
\]
Consequently $\ol{Q_a}\cong Q_a$ as Hodge structures and the theorem follows.
\end{proof}
\begin{coro}
\label{coro: main truncated comparison}
Let $a$ be coprime to $n$, take $\lambda=\rho/n$ and $\theta=e^{2\pi\rho\sqrt{-1}/n}$. Then in $\Coh_{\Nilp,G\times\C^\times}(\Cfrak)$,
\[
    \wt\CH\bigl(\gr_{\Hcal,\lambda}(\nabla_{c^a,\theta})\bigr) \cong \gr_{\mathcal{CS}}\left( \bigoplus_{z\in Z_G}z\Nbb_a(-m_a) \right).
\]
\end{coro}
Corollary \ref{coro: gr nabla_theta} identifies the left hand side as $\wt\CH\left((i_{c^a})_*\Ocal_{c^a}(-\omega_a+\rho,\rho)\right)\langle b-a(n-a)\rangle.$
\subsection{The $\GL_n$ lift}\label{SS:GL lift}
Firstly, for $\GL_n$, the length-zero subgroup is $\Z$ generated by $\varpi:=t_{\omega_1}c$ with
\begin{align}
    \label{eq: pi n central} \varpi^{\,n}=t_{\zeta},\qquad \zeta:=(1,1,\dots,1)\in X^*(T).
\end{align}
We first make the descent from $\SL_n$ explicit. Keep $\lambda=\rho/n$. Put
\[
    \mathcal K_{a,n}^{\SL}:=\gr_{\Hcal,\lambda}(\nabla_{c^a,\theta}),\qquad \ol{\Nbb}_{a,n}^{\SL}:= \bigoplus_{z\in\mu_n}z\Nbb_a.
\]
Consider the central isogeny $\pi:\SL_n\times\C^\times\longrightarrow\GL_n,\ (g,t)\longmapsto tg$ whose kernel is the antidiagonal $\mu_n$. Take $\ol{\Zcal}^R:= \wt{\Ncal}\times^R_{\gl_n}\wt{\Ncal}$. Under the decomposition $\mathfrak{gl}_n=\mathfrak{sl}_n\oplus\mathfrak z$, we have
\[
    \ol{\Zcal}^R=\Zcal^R\times \Spec(\Sym(\mathfrak z^*[1])).
\]
Let $q:\Zcal^R\to \ol{\Zcal}^R$ denote the map induced by the augmentation $\Sym(\mathfrak z^*[1])\to \C$.

Write $d=a+kn$ with $1\leq a\leq n-1$ and $a$ coprime to $n$ when $n>1$ and set $a=0$ when $n=1$. Define
\[
    \mathcal K_{d,n}:=q_*\Kcal_a^{\SL}\otimes\C_{\det^k},\qquad \ol{\Nbb}_{d,n}:=\pi_*\left(\ol{\Nbb}_{a,n}^{\SL}\boxtimes \Ocal_{\C^\times}\right)^{\chi_a}\otimes\C_{\det^k}.
\]
On $\mathcal K^\SL_{a,n}$, we define a scaling $\C^\times$-equivariance by letting $t\in\C^\times$ act on local sections as multiplication by $t^a$. Then the $\SL_n\times \C^\times$-equivariance on $\Kcal_{d,n}$ descends along $\pi$ to a $\GL_n$-equivariant structure. Similar for $\ol{\Nbb}_{d,n}$. Corollary \ref{coro: main truncated comparison} plus $\mu_n$-equivariance implies the following:
\begin{coro}\label{prop: GL lift}
In $\Coh_{\Nilp,\GL_n\times\C^\times}(\Cfrak_n)$ there is an isomorphism
\[
    \wt\CH_n(\mathcal K_{d,n}) \cong \gr_{\mathcal{CS}}\left( \ol{\Nbb}_{d,n}(-m_a) \right).
\]
\end{coro}
\section{Compatibility with parabolic induction}\label{S_pind_compat}
\subsection{Parabolic induction functors}
Here we consider the following situation. Let $G$ be a connected reductive group,  $P\subset G$ be a parabolic subgroup, and $L\subset P$ be a Levi subgroup. Choose a maximal torus $T\subset L$ and let $\theta$ have the same meaning as in Section \ref{SS_Hecke_monodromic}. We can form the following spaces and categories:
\begin{itemize}
\item[(i)] We can consider the derived commuting schemes $\mathfrak{C}_L,\mathfrak{C}_G$ for $L$ and $G$. Similarly define $\Cfrak_P$ to be the derived zero fiber of the  map $P\times \pfrak \to \pfrak, (p,x)\mapsto \operatorname{Ad}(p)x-x$. There is a correspondence
$$\mathfrak{C}_L/L\xleftarrow{\iota} \Cfrak_P/P \xrightarrow{\pi}\mathfrak{C}_G/G,$$
{  where $\iota$ is induced by the projection $P\twoheadrightarrow L$ and $\pi$ is induced by the inclusion $P\hookrightarrow G$.}
This correspondence gives rise to the pull-push functor as follows. 
We consider the functor
$${\pi_*\iota^*}: \Coh_{L\times \C^\times}(\Cfrak_L)\rightarrow
\Coh_{G\times \C^\times}(\Cfrak_G)$$ that restricts to
$$ \Coh_{\mathcal{N},L\times \C^\times}(\Cfrak_L)\rightarrow
\Coh_{\mathcal{N},G\times \C^\times}(\Cfrak_G).$$ We denote this functor by $\,_{\mathcal{C}}\widetilde{\operatorname{Ind}}_P^G$.

Similarly to Section \ref{SSS:affine_CH}, the functor $\,_{\mathcal{C}}\widetilde{\operatorname{Ind}}_P^G$ decomposes as the composition of two functors between categories of coherent sheaves on cotangent bundles introduced in Section \ref{SS_assoc_graded_general}. First we apply the pull-push along
$$T^*(L/\operatorname{Ad}(L))\leftarrow P/\operatorname{Ad}(P)\times_{L/\operatorname{Ad}(L)}T^*(L/\operatorname{Ad}(L))\rightarrow T^*(P/\operatorname{Ad}(P))$$
that is associated to the smooth morphism $P/\operatorname{Ad}(P)\rightarrow L/\operatorname{Ad}(L)$ and then the twisted (by the canonical bundle of $P/\operatorname{Ad}(P)$) pull-push via
$$T^*(P/\operatorname{Ad}(P))\leftarrow P/\operatorname{Ad}(P)\times_{G/\operatorname{Ad}(G)}T^*(G/\operatorname{Ad}(G))\rightarrow T^*(G/\operatorname{Ad}(G)).$$
It is associated to the projective morphism $P/\operatorname{Ad}(P)\rightarrow G/\operatorname{Ad}(G)$.
\item[(i$'$)] Similarly, we have the parabolic induction functor
$$\,_{\mathcal{C}}\operatorname{Ind}_P^G: D^b_{L}(L)\rightarrow D^b_{G}(G),$$
{where we abuse the notation and write $D^b(G)$ for $D^b_G\operatorname{MHM}(G)$, and similarly to $D^b_L(L)$.}
\item[(ii)] We can form the Steinberg derived schemes $\mathcal{Z}_L,\mathcal{Z}_G$, as well as the (derived) affine Hecke categories $\mathcal{AH}_L,\mathcal{AH}_G$. We can view these categories as the dg categories $\operatorname{Coh}_{\C^\times}(T^*(B_L\backslash L/B_L)),\operatorname{Coh}_{\C^\times}(T^*(B\backslash G/B))$, where $B_L:=L\cap B$.
We can then, as in (i), define the composition $\,_{\mathcal{H}}\widetilde{\operatorname{Ind}}_P^G$ of pull-push functors associated to stack morphisms
$$B_L\backslash L/B_L\leftarrow B\backslash P/B\rightarrow B\backslash G/B.$$
\item[(ii$'$)] Similarly to (i$'$), we have the categories $\mathcal{H}_{L,\theta}, \mathcal{H}_{G,\theta}$ and the induction functor $$\,_{\mathcal{H}}\operatorname{Ind}_P^G:\mathcal{H}_{L,\theta}\rightarrow \mathcal{H}_{G,\theta}.$$
\end{itemize}
\begin{rem}
If $Z_G$ has positive dimension (for instance, when $G=\GL_n$), the categories above should be understood within the $\chi$-twisted equivariant framework introduced in Remark \ref{rem:twisted_CH}. The character $\chi:\gfrak\to \C$ restricts naturally along $\mathfrak{l}\subset \pfrak\subset \gfrak$, and the four induction operations preserve this twist; hence, we omit explicit mention of $\chi$.
\end{rem}
\subsection{Compatibility of functors}

\begin{lem}\label{Lem:commuting_cube}
For all $\theta$, the functor $\,_{\mathcal{C}}\operatorname{Ind}_P^G$ sends $\mathcal{CS}_{L,\theta}$ to $\mathcal{CS}_{G,\theta}$. Moreover, the following diagram is commutative, where
\begin{enumerate}
\item the vertical arrows are the functors of taking the Hodge associated graded sheaf;
\item the horizontal arrows are the character functors $\operatorname{CH}$ or $\widetilde{\operatorname{CH}}$;
\item and the slant arrows are parabolic induction functors $\,_?\operatorname{Ind}_P^G$ or
$\,_?\widetilde{\operatorname{Ind}}_P^G$.
\end{enumerate}
\begin{equation}
    \label{eq:commuting_cube} \begin{picture}(130,70) \put(7,57){$\Hcal_{L,\theta} $} \put(7,27){$\AH_L$} \put(67,57){$\CS _{L,\theta}$} \put(47,27){$\operatorname{Coh}_{L\times \C^\times, \Nilp}(\Cfrak_L)$} \put(27,37){$\Hcal_{G,\theta} $} \put(27,7){$\AH_G$} \put(87,37){$\CS _{G,\theta}$} \put(67,7){$\operatorname{Coh}_{G\times \C^\times, \Nilp}(\Cfrak_G)$} \put(16,58){\vector(1,0){50}} \put(36,38){\vector(1,0){50}} \put(16,28){\vector(1,0){30}} \put(36,8){\vector(1,0){30}} \put(11,55){\vector(0,-1){24}} \put(31,35){\vector(0,-1){24}} \put(71,55){\vector(0,-1){24}} \put(91,35){\vector(0,-1){24}} \put(14,55){\vector(1,-1){14}} \put(74,55){\vector(1,-1){14}} \put(14,25){\vector(1,-1){14}} \put(74,25){\vector(1,-1){14}}
    \end{picture}
\end{equation}
\end{lem}
\begin{proof}
1) Induction vs character functors.

We start with the commutativity of the version of the top square of (\ref{eq:commuting_cube}), where instead of the categories of character sheaves we have $D^b_L(L)$ and $D^b_G(G)$. Consider the following commutative diagram of quotient stacks:
\begin{equation}
    \label{eq:diagram_pull_push} \begin{picture}(80,40) \put(19,2){$Y_L$} \put(12,17){$L/\operatorname{Ad}B_L$} \put(12,32){$L/\operatorname{Ad}L$} \put(47,2){$Y_P$} \put(42,17){$P/\operatorname{Ad}B$} \put(42,32){$P/\operatorname{Ad}P$} \put(76,2){$Y_G$} \put(72,17){$G/\operatorname{Ad}B$} \put(72,32){$G/\operatorname{Ad}G$} \put(20,21){\vector(0,1){10}} \put(20,16){\vector(0,-1){10}} \put(48,21){\vector(0,1){10}} \put(48,16){\vector(0,-1){10}} \put(78,21){\vector(0,1){10}} \put(78,16){\vector(0,-1){10}} \put(45,3){\vector(-1,0){20}} \put(53,3){\vector(1,0){21}} \put(41,18){\vector(-1,0){12}} \put(58,18){\vector(1,0){13}} \put(41,33){\vector(-1,0){14}} \put(58,33){\vector(1,0){13}} \end{picture}
\end{equation}
where $$Y_L=(U_L\backslash L/U_L)/T, Y_P=(U\backslash P/U)/T, Y_G=(U\backslash G/U)/T.$$ We note that
\begin{itemize}
\item[(i)] The composition $\,_{\mathcal{C}}\operatorname{Ind}_P^G\circ \operatorname{CH}_L$ is the composition of the pull-push along the leftmost column followed by the pull-push along the top row.
\item[(ii)] The composition
$\operatorname{CH}_G\circ \,_{\mathcal{H}}\operatorname{Ind}_P^G$ is the composition of the pull-push along the bottom row followed by the pull-push along the right column.
\item[(iii)] The top-left and bottom-right squares in the diagram are Cartesian.
\item[(iv)] The morphisms $Y_P\rightarrow Y_G$ and $L/\operatorname{Ad}B_L\rightarrow L/\operatorname{Ad}L$ are proper.
\end{itemize}
So we can apply proper base change and conclude that the functors in (i) and (ii) are both isomorphic to the pull-push via
$$Y_L\leftarrow P/\operatorname{Ad}B\rightarrow G/\operatorname{Ad}G$$
This shows the commutativity of the top square of (\ref{eq:commuting_cube}) modified as in the beginning of the proof and also that $\,_{\mathcal{C}}\operatorname{Ind}_P^G$ sends objects in the image of $\CH_L$ to $\CS_{G,\theta}$. {For the assertion about the image of $\CS_{L,\theta}$ under $\,_{\mathcal{C}}\operatorname{Ind}_P^G$, by \cite{BYD}, $\,_{\mathcal{C}}\operatorname{Ind}_P^G$ is $t$-exact. Exactness then preserves subquotients and extensions and hence the Serre subcategory $\CS_{L,\theta}^\heartsuit$ and implies that $\,_{\mathcal{C}}\operatorname{Ind}_P^G\CS_{L,\theta}^\heartsuit$ lands in $\CS_{G,\theta}$.}

Now we proceed to the commutativity of the bottom square of (\ref{eq:commuting_cube}). In order to establish this we need the following commutative diagram completely analogous to (\ref{eq:diagram_pull_push}):

    \begin{picture}(80,40)
		\put(2,2){$B_L\backslash L/B_L$}
		\put(2,17){$L/\operatorname{Ad}B_L$}
		\put(2,32){$L/\operatorname{Ad}L$}
		\put(32,2){$B\backslash P/B$}
		\put(32,17){$P/\operatorname{Ad}B$}
		\put(32,32){$P/\operatorname{Ad}P$}
		\put(62,2){$B\backslash G/B$}
		\put(62,17){$G/\operatorname{Ad}B$}
		\put(62,32){$G/\operatorname{Ad}G$}
		\put(10,21){\vector(0,1){10}}
		\put(10,16){\vector(0,-1){10}}
		\put(38,21){\vector(0,1){10}}
		\put(38,16){\vector(0,-1){10}}
		\put(68,21){\vector(0,1){10}}
		\put(68,16){\vector(0,-1){10}}
		\put(31,3){\vector(-1,0){11}}
		\put(46,3){\vector(1,0){15}}
		\put(31,18){\vector(-1,0){12}}
		\put(48,18){\vector(1,0){13}}
		\put(31,33){\vector(-1,0){12}}
		\put(48,33){\vector(1,0){13}}
	\end{picture}

Note that this diagram satisfies analogs of (iii) and (iv) above. So we can apply the proper base change for functors between categories of coherent sheaves on cotangent bundles, cf. Section \ref{SSS:affine_CH}. This yields the commutativity of the bottom square of (\ref{eq:commuting_cube}) following the scheme explained above for the D-module functors, cf. \cite[Lemma 4.6]{BezLosev}.

2) Character vs Hodge associated graded functors: this is the content of Theorem \ref{thm: grCH}.

3) Induction vs Hodge associated graded functors: this is completely analogous to 2) and is left as an exercise.
\end{proof}

\subsection{Monoidal structure}\label{SS_monoidal}
We consider the following situation. Let 
$Q\subset G$ be a parabolic containing $P$ and $M\subset Q$ a Levi subgroup containing $L$. Write $P_M$ for $P\cap M$, this is a parabolic subgroup. Proper base change yields the following isomorphisms of functors
\begin{equation}
    \label{eq:transitivity_iso} \psi_{Q,P}^G:\,_{\mathcal{C}}\widetilde{\operatorname{Ind}}_Q^G\circ \,_{\mathcal{C}}\widetilde{\operatorname{Ind}}_{P_M}^M\xrightarrow{\sim}\,_{\mathcal{C}} \widetilde{\operatorname{Ind}}_P^G.
\end{equation}

Now suppose we are given three embedded parabolic subgroups $P\subset Q\subset R\subset G$ with Levi subgroups $L\subset M\subset K$. The notation like $P_K$ is used for the intersection like $K\cap P$. We have the following transitivity result.

\begin{lem}\label{lem:transitivity_iso}
We have an equality
$$\psi^G_{R,P}\circ (\mathsf{1}\psi^K_{Q_K,P_K})=
\psi^G_{Q,P}\circ (\psi^G_{R,Q}\mathsf{1}).$$
of functor morphisms
$$\,_{\mathcal{C}}\widetilde{\operatorname{Ind}}_R^G\circ \,_{\mathcal{C}}\widetilde{\operatorname{Ind}}_{Q_K}^K\circ \,_{\mathcal{C}}\widetilde{\operatorname{Ind}}_{P_M}^M\xrightarrow{\sim}\,_{\mathcal{C}}
\widetilde{\operatorname{Ind}}_P^G$$ Here $\mathsf{1}$ denotes the identity endomorphisms of the appropriate functors and the notation like $\mathsf{1}\psi^K_{Q_K,P_K}$ denotes the natural transformation between the corresponding composite functors.
\end{lem}
\begin{proof}
Consider the following commutative diagram.
\[
    \begin{tikzcd}[column sep=tiny,row sep=small] &&& \Cfrak_P/P &&& \\ && \Cfrak_{P_K}/P_K && \Cfrak_Q/Q \\ & \Cfrak_{P_M}/P_M && \Cfrak_{Q_K}/Q_K && \Cfrak_R/R \\ \Cfrak_L/L && \Cfrak_M/M && \Cfrak_K/K && \Cfrak_G/G \arrow[from=1-4, to=2-3] \arrow[from=1-4, to=2-5] \arrow[from=2-3, to=3-2] \arrow[from=2-3, to=3-4] \arrow[from=2-5, to=3-4] \arrow[from=2-5, to=3-6] \arrow[from=3-2, to=4-1] \arrow[from=3-2, to=4-3] \arrow[from=3-4, to=4-3] \arrow[from=3-4, to=4-5] \arrow[from=3-6, to=4-5] \arrow[from=3-6, to=4-7] \end{tikzcd}
\]
The three squares are derived Cartesian: take the upper square as an example. We have that $\Cfrak_P\cong (P\times \pfrak)\times^R_\pfrak\{0\}$ and
\begin{align}
    \label{iso: Comm} \Cfrak_P/\Ad(P)\cong [(P\times \pfrak)/\Ad(P) ]\times^R_{[\pfrak/\Ad(P)]} \mathsf{B}P,
\end{align}
{  where we write $\mathsf{B}P$ for $\operatorname{pt}/P$,} while $(P\times \pfrak)/\Ad(P)\cong P/\Ad(P)\times_{\mathsf{B}P} \pfrak/\Ad(P)$.

Let $\qfrak:=\Lie(Q)$. We have {  isomorphisms} $P\cong P_K\times_{Q_K} Q$ and $\pfrak\cong \pfrak_K\times_{\qfrak_K}\qfrak$. Since the maps $Q\to Q_K$ and $\qfrak\to \qfrak_K$ are smooth and surjective, { we have $P_K\times_{Q_K} Q\cong P_K\times_{Q_K}^R Q, \pfrak_K\times_{\qfrak_K}\qfrak=\pfrak_K\times^R_{\qfrak_K}\qfrak$}. We also have $\mathsf{B}P\cong \mathsf{B}P_K\times^R_{\mathsf{B}Q_K}\mathsf{B}Q$. Moreover, taking loop stacks, resp. shifted tangent stack, preserves derived fiber products and gives $L(\mathsf{B}P)\cong [P/\Ad(P)]$, resp. $T_{\mathsf{B}P}[-1]\cong [\pfrak/\Ad(P)]$. Hence we also have
\[
    [P/\Ad(P)]\cong [P_K/\Ad(P_K)]\times^R_{[Q_K/\Ad(Q_K)]} [Q/\Ad(Q)],
\]
\[
    [\pfrak/\Ad(P)]\cong [\pfrak_K/\Ad(P_K)]\times^R_{[\qfrak_K/\Ad(Q_K)]} [\qfrak/\Ad(Q)].
\]
Finally, given the description \eqref{iso: Comm}, commutation of finite limits proves the upper square is derived Cartesian.

The isomorphisms $\psi^G_{R,P}\circ (\mathsf{1}\psi^K_{Q_K,P_K})$ and $\psi^G_{Q,P}\circ (\psi^G_{R,Q}\mathsf{1})$ are compositions of base changes through the three derived Cartesian squares in the diagram above, and therefore are equal.
\end{proof}

Now suppose $G=\operatorname{GL}_n$. We write $G_n$ for $\operatorname{GL}_n, \mathfrak{C}_n$ for the corresponding commuting derived scheme, and $\widetilde{\mathcal{C}}_n$ for $\operatorname{Coh}_{\mathcal{N},G_n\times \C^\times}(\mathfrak{C}_n)$. Our task is to equip $\bigoplus_{n\geq 0}\widetilde{\mathcal{C}}_n$ with a monoidal structure. For a composition $\vec{n}=(n_1,\ldots,n_k)$ we write $P_{\vec{n}}$ for the stabilizer of the standard partial flag $\C^{n_1}\subset \C^{n_1+n_2}\ldots \C^{n_1+\ldots+n_i}\subset\ldots \subset \C^n$ where $n=\sum_i n_i$.

1) The tensor product functor $\star: \wt{\Ccal}_n\boxtimes \wt{\Ccal}_m\rightarrow \wt{\Ccal}_{n+m}$ is given by $\,_{\mathcal{C}}\wt{\Ind}^{G_{n+m}}_{P_{(n,m)}}$.

2) The associativity isomorphism
$$(\mathcal{F}_{1}\star \mathcal{F}_2)\star \mathcal{F}_3\xrightarrow{\sim} \mathcal{F}_1\star (\mathcal{F}_2\star \mathcal{F}_3), \mathcal{F}_i\in \widetilde{\mathcal{C}}_{n_i},$$
is given by $$(\psi^n_{(n_1,n_2+n_3),\vec{n}})^{-1}\circ\psi^n_{(n_1+n_2,n_3),\vec{n}},$$ where $\vec{n}=(n_1,n_2,n_3)$ and the compositions in the subscript indicate the corresponding parabolics.

3) The monoidal unit is the object $\C$ in $\wt{\Ccal}_0=D^b\Rep(\C^\times)$.

4) The pentagon axiom is a consequence of Lemma \ref{lem:transitivity_iso}: under the identification of all terms involved with $\operatorname{Ind}^n_{\vec{n}}(\mathcal{F}_1\boxtimes \mathcal{F}_2\boxtimes \mathcal{F}_3\boxtimes \mathcal{F}_4)$ all isomorphisms in the pentagon diagram become the identity, as the lemma shows that these isomorphisms do not depend on the way to bracket.

\begin{rem}
Note that we have analogs of (\ref{eq:transitivity_iso}) for the affine Hecke categories for $\operatorname{GL}_n$. We write $\mathcal{AH}_n$ for the affine Hecke category for $\operatorname{GL}_n$. Direct analogs of Lemma \ref{lem:transitivity_iso} can be used to show that
\begin{itemize}
\item $\bigoplus_{n\geq 0}\mathcal{AH}_n$ is a monoidal category with respect to the parabolic induction.
\item And $\bigoplus_{n\geq 0}\widetilde{\operatorname{CH}}$ is a monoidal functor.
\end{itemize}
\end{rem}


\section{Symmetric group actions}\label{S_sym_group_act}
Suppose now $n=dr$ for $d,r\in\Z_{>0}$. Set $L=\GL_r^d\subset G=\GL_n$ with an associated parabolic $P$. Recall the $D$-module $\Nbb_{a,r}$ on $\GL_r$ from Section \ref{SS:GL lift}. 
We may consider the convolution power
\begin{align}
    \label{def: N conv d} \olNbb_{a,r}^{\star d} := {}_{\Ccal}\Ind_P^G(\olNbb_{a,r}^{\boxtimes d})
\end{align}
for $1\leq a\leq r-1$ coprime to $r$ when $r>1$ and $a=0$ when $r=1$. The main result in this section is the following.
\begin{thm}\label{prop_Sd_action}
There is an injection $\C S_d\hookrightarrow \End_{\MHM(G)}\bigl(\ol\Nbb_{a,r}^{\star d}\bigr)$.
\end{thm}

Note that when $r=1$, $\olNbb_{0,1}^{\star n}$ is the group version Grothendieck-Springer $D$-module, and the result is standard. Our construction generalizes this case and follows a similar proof strategy as in \cite{HK}.

To set up the proof, let $W_L=S_r^d$ be the Weyl group of the Levi $L=\GL_r^d$; the relative Weyl group $N_W(W_L)/W_L$ is $S_d$, permuting the $d$ blocks, and it is this $S_d$ that will act below. Take $T_d=Z(L)$ and put $ A_d:=P/[P,P]\cong L/[L,L]. $ Let $ \nu:T_d\longrightarrow A_d $ be the natural isogeny. Let
\[
    T_d^\circ := \{t\in T_d\mid \Stab_W(t)=W_L,\ \Stab_{S_d}(\nu(t))=\{1\}\}
\]
and set $A_d^\circ:=\nu(T_d^\circ)$.

Let $U_P$ be the unipotent radical of $P$, and let $O_{(r)}\subset\GL_r$ be the principal unipotent orbit. Take
\[
    \wt G_d^\circ := G\times^P \bigl(T_d^\circ\cdot O_{(r)}^d\cdot U_P\bigr).
\]
There is an isomorphism
\begin{align}
    \label{iso: UTO} U_P \times T_d \times O_{(r)}^d\to T_d\cdot O_{(r)}^d\cdot U_P,\ (u,t,x)\mapsto txu.
\end{align}
Define $G_d^\circ\subset G$ to be the image of $\wt G_d^\circ$ under $ G\times^P P\longrightarrow G,\ [g,p]\longmapsto g\cdot p$, { where in the homogeneous bundle the action of $P$ on itself is adjoint}. There is an embedding
\[
    \wt G_d^\circ\longrightarrow G\times A_d,\qquad [g,tuv]\longmapsto\bigl(g\cdot (tuv),\nu(t)\bigr)
\]
Moreover, $\wt G_d^\circ \cong G_d^\circ\times_{A_d^\circ/S_d}A_d^\circ, $ so the induced map $\wt G_d^\circ\to G_d^\circ$ is an $S_d$-Galois cover. Set $G_d$ to be the closure of $G_d^\circ$ in $G$.

Recall that $\pi_1(O_{(r)})=\Z/r\Z$. Let $\Kbb_a$ be the integrable connection on $O_{(r)}$ corresponding to $e^{2\pi a\sqrt{-1}/r}\in (\Z/r\Z)^*$. Let $\wt{\Ebb}_a^e$ be the connection on $T_d\cdot O_{(r)}^d\cdot U_P$ pulled back from the connection $\Kbb_a^{\boxtimes d}$ on $O_{(r)}^d$ with respect to \eqref{iso: UTO}. Put $S(P)=P\cap\SL_n$ and $S(L)=L\cap \SL_n$. Each $\Kbb_a$  is $GL_r$-twisted equivariant. Moreover, $Z(\GL_r)$ acts on $\Kbb_a$ by $z\mapsto z^a$. Therefore, $\Kbb_a^{\boxtimes d}$ is $S(L)$-equivariant.

Let $ j:T_d\cdot O_{(r)}^d\cdot U_P\hookrightarrow P $ be the locally closed inclusion. In each $\GL_r$ block, the restriction of $\wt{\Ebb}_a^e$ is clean as in Section \ref{S_traces}. Therefore, $ j_!\wt{\Ebb}_a^e \cong j_{!*}\wt{\Ebb}_a^e \cong j_\dagger\wt{\Ebb}_a^e, $ and we denote this object by $\wt{\Mbb}_a^e$. Let $\wt{\Mbb}_a = \Ind_{S(P)}^{\SL_n}\wt{\Mbb}_a^e$, { this is a D-module on $G\times^P P$}. It is also the $\dagger$-extension of $ \wt{\Ebb}_a := \Ind_{S(P)}^{\SL_n}\wt{\Ebb}_a^e$. By construction, $\ol\Nbb_{a,r}^{\star d}$ is the pushforward of $\wt{\Mbb}_a$ to $G$.

{  Let $\Unip_L$ denote the unipotent variety of $L$, $\Nilp_L$ denote the nilpotent cone in $\mathfrak l$.} Under the trace form, we identify $\gfrak\cong \gfrak^*$ as well as $\pfrak^*\cong \pfrak_-$ etc.

\begin{lem}\label{lem: SS M tilde}
We have $ SS(\wt{\Mbb}_a)=G\times^P\Sigma, $ where
	\[
    \Sigma
    = \{(tuv,x)\mid t\in T_d,\ u\in\Unip_L,\ v\in U_P,\ x\in\Nilp_L,\ u\cdot x=x\}.
\]
\end{lem}

\begin{proof}
First of all, by \cite[Theorem 5(b)]{Lusztig_fourier_cuspidal}, $SS(\Nbb_{a,r}) = \{(u,x)\mid \ u\in\Unip_{\GL_r},\ x\in\Nilp_{\GL_r},\  u\cdot x=x\}.$ As a result,
	\[
    SS(\ol{\Nbb}_{a,r})
    = \{(zu,x)\mid z\in\C^\times,\ u\in\Unip_{\GL_r},\ x\in\Nilp_{\GL_r},\  u\cdot x=x\}.
\]
Hence by the block-wise construction, $SS(\wt{\Mbb}_a^e)=\Sigma$. The lemma now follows from Lemma \ref{prop_gr_induction}.
\end{proof}

\subsection{Generalized Springer $D$-module}

Let $ q_P:P\longrightarrow A_d=P/[P,P] $ be the quotient map and define
\[
    f:G\times^P P\longrightarrow G\times A_d,\qquad [g,h]\longmapsto\bigl(g\cdot h,q_P(h)\bigr).
\]
In particular, if $h=tuv$ with $t\in T_d$, $u\in\Unip_L$, and $v\in U_P$, then $q_P(h)=\nu(t)$.

Define $ \Mbb_a:=f_\dagger\wt{\Mbb}_a, $ which resembles the Harish-Chandra $D$-module of Hotta--Kashiwara \cite{HK}. Let $ \tfrak_d:=\Lie(T_d)$ so that $\pfrak$ has a splitting $\tfrak_d\oplus [\pfrak,\pfrak]$ and hence $\Lie(A_d)\cong\tfrak_d$.

We identify
\[
    T^*(G\times A_d) \cong G\times A_d\times\gfrak\times\tfrak_d.
\]

\begin{lem}\label{lem: Lambda irr}
The singular support of $\Mbb_a$ is $\Lambda:=\ol{G\cdot\Lambda_L}$, where
	\[
    \Lambda_L := \{(tu,\nu(t),x+y,-y)\mid t\in T_d,\ u\in\Unip_L,\ x\in\Nilp_L,\ y\in\tfrak_d,\  u\cdot x=x\}.
\]
Moreover, $\Lambda$ is irreducible and Lagrangian in $T^*(G\times A_d)$.
\end{lem}

\begin{proof}
Take
	\[
    i:G\times^P P\longrightarrow G/P\times G\times A_d, \qquad [g,h]\longmapsto(gP,g\cdot h,q_P(h)).
\]
Recall the maps $\rho_i$ and $\varpi_i$ from \eqref{diagram_lagrangian}. Since $i$ is a closed embedding,
	\[
    SS(i_\dagger\wt{\Mbb}_a)
    = \varpi_i\rho_i^{-1}\bigl(SS(\wt{\Mbb}_a)\bigr).
\]

By $G$-equivariance, it suffices to calculate { the fiber of } $SS(i_\dagger\wt{\Mbb}_a)$ over $[1,h]$. For $\alpha\in\gfrak$ and $\beta\in\pfrak$, we have
	\[
    di_{[1,h]}(\alpha,\beta)
    = \left( \alpha\bmod\pfrak,\ h^{-1}\cdot \alpha-\alpha+\beta,\ dq_P(\beta) \right).
\]
Suppose that
	\[
    (x,y,z) \in T^*_{i([1,h])}(G/P\times G\times A_d)
\]
annihilates the image of $di_{[1,h]}$. Then
	\[
    \langle x,\alpha\rangle + \langle y, h^{-1}\cdot \alpha-\alpha+\beta\rangle + \langle z,dq_P(\beta)\rangle
    =0
\]
for every $\alpha\in\gfrak$ and $\beta\in\pfrak$. It follows that $x=y- h\cdot y$,
	\[
    y\in[\pfrak,\pfrak]^\perp
    = \tfrak_d+\nfrak_P, \qquad z=-\ol{y}
\]
where $\ol{y}$ is the projection of $y$ to $\mathfrak{t}_d\cong \mathfrak{p}/[\mathfrak{p},\mathfrak{p}]$ Hence
	\[
    \Ker(d^*i_{[1,h]})
    = \{(y- h\cdot y,y,-\ol{y})\mid y\in\tfrak_d+\nfrak_P\}.
\]

Given Lemma \ref{lem: SS M tilde}, we obtain
	\begin{align*}
    SS(i_\dagger\wt{\Mbb}_a)|_{i([1,tuv])}
    = \{(&x+y- tuv\cdot (x+y),x+y,-\ol{y})\mid\\ &x\in\Nilp_L,\  u\cdot x=x,\ y\in\tfrak_d+\nfrak_P,\ t\in T_d,\ u\in\Unip_L,\ v\in U_P\}.
\end{align*}

Note that $(x+y)- tuv\cdot (x+y)\equiv x- u\cdot x \pmod{\nfrak_P}$. Take $P_d:=T_d\Unip_LU_P$ and $\pfrak_d:=\tfrak_d+\Nilp_L+\nfrak_P$, and let $pr:G/P\times G\times A_d\longrightarrow G\times A_d$ be the projection. Then
	\[
    \rho_{pr}^{-1}\bigl(SS(i_\dagger\wt{\Mbb}_a)\bigr) \cong G\times^P\Lambda_P,
\]
where
	\[
    \Lambda_P
    = \{(h,q_P(h),z,- \overline{z})\in P_d\times A_d\times\pfrak_d\times\tfrak_d \mid h\cdot z=z\}.
\]
Therefore, $SS(\Mbb_a) \subseteq \ol{G\cdot\Lambda_P}.$

We claim that $\ol{G\cdot\Lambda_P}=\ol{G\cdot\Lambda_L}$. The inclusion $\Lambda_L\subseteq\Lambda_P$ is immediate. Conversely, take $(h,q_P(h),z,-\ol z)\in\Lambda_P$. Suppose that the commuting pair $(h,z)$ has distinct joint eigenvalue pairs $(\xi_i,\eta_i)_{i=1}^k$, with corresponding generalized eigenspaces of dimensions $rm_i$, where $\sum_{i=1}^k m_i=d$. After replacing $(h,z)$ by a $G$-conjugate, we may assume that
	\[
    (h,z)\in \left(\prod_{i=1}^k\GL_{rm_i}\right) \times \left(\prod_{i=1}^k\mathfrak{gl}_{rm_i}\right).
\]
Let $h_i$ and $z_i$ be the $i$-th blocks of $h$ and $z$, respectively. Then $ (\xi_i^{-1}h_i-I,z_i-\eta_iI) $ is a pair of commuting nilpotent matrices.

By \cite{Baranovsky}, the variety of pairs of commuting nilpotent matrices is irreducible, and the open locus on which the first matrix is regular nilpotent is dense. To show $\Lambda_P\subset \ol{G\cdot \Lambda_L}$, it therefore suffices to consider
	\[
    \left(\xi_i^{-1}h_i-I,z_i-\eta_iI\right)
    = \left(J_{rm_i},p_i(J_{rm_i})\right),
\]
where $J_{rm_i}$ is a regular nilpotent Jordan block and $p_i$ is a polynomial satisfying $p_i(0)=0$.

The matrix $J_{rm_i}$ lies in the closure of the regular locus with $m_i$ distinct eigenvalues and one Jordan block of size $r$ over each eigenvalue. Deforming $J_{rm_i}$ within this locus and simultaneously deforming $p_i(J_{rm_i})$ by applying the same polynomial, we obtain commuting pairs lying in $G\cdot\Lambda_L$. Hence $\Lambda_P\subseteq\ol{G\cdot\Lambda_L}$. Together with the reverse inclusion, this proves $\ol{G\cdot\Lambda_P}=\ol{G\cdot\Lambda_L}(=\Lambda)$. Consequently, $SS(\Mbb_a)\subseteq\Lambda$.

Now
	\[
    \Lambda_L \cong T_d\times\tfrak_d\times \{(u,v)\in\Unip_L\times\Nilp_L\mid u\cdot v=v\}.
\]
The map $u\mapsto u-I$ identifies the last factor with a product of varieties of pairs of commuting nilpotent matrices. Therefore $\Lambda_L$ is irreducible by \cite{Baranovsky}, and it follows that $\Lambda=\ol{G\cdot\Lambda_L}$ is irreducible.

Finally, on the dense smooth locus $\wt G_d^\circ\subset G\times A_d$, the restriction of $\Mbb_a$ is a nonzero integrable connection. Therefore $ \ol{T^*_{\wt G_d^\circ}(G\times A_d)} \subseteq SS(\Mbb_a). $ At a point $(tu,\nu(t))\in\wt G_d^\circ$, where $t\in T_d^\circ$ and $u\in O_{(r)}^d$, we have
	\[
    T^*_{\wt G_d^\circ,(tu,\nu(t))}(G\times A_d)
    = \{(x+y,-y)\mid x\in\Nilp_L,\  u\cdot x=x,\ y\in\tfrak_d\}.
\]
This is precisely the restriction of $\Lambda$ to the dense locus where $t\in T_d^\circ$ and every component of $u$ is regular unipotent. Hence $\Lambda=\ol{T^*_{\wt G_d^\circ}(G\times A_d)} \subseteq SS(\Mbb_a)$. Combining this with the reverse inclusion proves $SS(\Mbb_a)=\Lambda$. Moreover, $\Lambda$ is Lagrangian because it is the closure of the conormal bundle to the smooth locally closed subvariety $\wt G_d^\circ$.
\end{proof}
For a connection $\Lcal$ on a smooth locally closed subvariety $U\subset X$, we write $IC(X,\Lcal)$ for its intermediate extension to $X$, supported on $\overline U$.

Let $\wt{\Ebb}_a^\circ$ denote the restriction of $\wt{\Ebb}_a$ to $\wt G_d^\circ$.
\begin{coro}\label{coro: M IC}
We have $\Mbb_a\cong IC(G\times A_d,\wt{\Ebb}_a^\circ)$.
\end{coro}

\begin{proof}

For $i\neq0$, the holonomic $D$-module $\Hcal^i\Mbb_a$ is supported on $\ol{\wt G_d^\circ}\setminus\wt G_d^\circ$. Its singular support is contained in a proper closed subvariety of $\Lambda$ of dimension less than $\dim\Lambda$, by Lemma \ref{lem: Lambda irr}. On the other hand, the singular support of every nonzero holonomic $D$-module on $X$ has pure dimension $\dim X$. Hence $\Hcal^i\Mbb_a=0$ for all $i\neq0$.

Set $M=\Hcal^0\Mbb_a$. If $K$ were a nonzero submodule or quotient of $M$ supported on $\ol{\wt G_d^\circ}\setminus\wt G_d^\circ$, then $SS(K)$ would again be contained in a proper closed subvariety of $\Lambda$, resulting in the same contradiction as above. Thus $M$ has no nonzero subobject or quotient supported on the boundary.

This is precisely the characterization of the intermediate extension, and the corollary follows.
\end{proof}

Let $\Ebb_a^\circ$ be the pushforward of $\wt{\Ebb}_a^\circ$ to $G_d^\circ$ along the $S_d$-\'etale map $q^\circ:\wt G_d^\circ\longrightarrow G_d^\circ.$

\begin{prop}\label{prop: IC N^d}
We have $\ol\Nbb_{a,r}^{\star d}\cong IC(G,\Ebb_a^\circ)$.
\end{prop}

\begin{proof}
The closure $\ol{\wt G_d^\circ}$ inside $G\times A_d$ lies inside $G_d\times_{A_d/S_d}A_d.$ Hence the projection $q:\ol{\wt G_d^\circ}\longrightarrow G_d$ is finite. Since finite pushforward is exact, it commutes with intermediate extension. Using Corollary \ref{coro: M IC}, we then have
	\[
    \ol\Nbb_{a,r}^{\star d}
    = q_*IC\!\left(G\times A_d,\wt{\Ebb}_a^\circ\right) \cong IC\bigl(G,q^\circ_*\wt{\Ebb}_a^\circ\bigr)
    = IC(G,\Ebb_a^\circ).\qedhere
\]
\end{proof}

\begin{proof}[Proof of Theorem \ref{prop_Sd_action}]
We first show there is an injection $\C S_d\hookrightarrow \End_{\mathrm{VHS}}(\Ebb_a^\circ).$ Let $(\wt{\Ebb}_a^\circ)^e$ be the connection on $T_d^\circ\cdot O^d_{(r)}$ pulled back from the connection $\Kbb_a^{\boxtimes d}$ on $O^d_{(r)}$. Then as discussed earlier, $(\wt{\Ebb}_a^\circ)^e$ is $S(L)$-equivariant. Note that $\wt{G}_d^\circ\cong G\times^L(T_d^\circ\cdot O^d_{(r)})$. Hence we have $\wt{\Ebb}_a^\circ=\Ind_{S(L)}^{\SL_n}(\wt{\Ebb}_a^\circ)^e$.  Now $(\wt{\Ebb}_a^\circ)^e$ is equivariant under the permutation $S_d$-action on the $d$-blocks. Let $N(S(L))$ be the normalizer of $S(L)$ in $\SL_n$. Then $(\wt{\Ebb}_a^\circ)^e$ is further $N(S(L))$-equivariant. Therefore, $\wt{\Ebb}_a^\circ$ is $S_d$-equivariant for the deck action. Hence $S_d$ acts on $\Ebb_a^\circ$. Since $S_d$ acts on every stalk of $\Ebb_a^\circ$ via its regular action, the action is faithful and we obtain the embedding $\C[S_d]\hookrightarrow \End_{\mathrm{VHS}}(\Ebb_a^\circ)$.

Now we use Proposition \ref{prop: IC N^d}. Since intermediate extension is fully faithful, the theorem follows.
\end{proof}

For any $\tau\vdash d$, let $V_\tau$ be the corresponding irreducible module of $S_d$. Let $\Ebb_{a,\tau} = \Hom_{S_d}(V_\tau,\Ebb_a^\circ).$ Then, under the $S_d$-action, $\ol\Nbb_{a,r}^{\star d}$ decomposes as
\[
    \ol\Nbb_{a,r}^{\star d} \cong \bigoplus_{\tau\vdash d} V_\tau\otimes IC(G,\Ebb_{a,\tau}),
\]
where each summand is viewed on $G$ by closed pushforward.

\begin{coro}\label{coro: gr(convolution)}
Under the $S_d$-action from Theorem \ref{prop_Sd_action}, $\gr(\ol\Nbb_{a,r}^{\star d})$ decomposes as
	\[
    \gr(\ol\Nbb_{a,r}^{\star d}) \cong \bigoplus_{\tau\vdash d} V_\tau\otimes\gr(IC(G,\Ebb_{a,\tau})).
\]
\end{coro}
\subsection{Lie algebra case}\label{SS:Lie algebra case}
Let $\gfrak=\gl_n$. The notations $\gfrak_d^\circ$, $\wt{\gfrak}_d^\circ$, $\Ebb_{a,\tau}$, etc. have meaning similar to the corresponding notations for $G$.

There is an analogue of Corollary \ref{coro: gr(convolution)} in this setting, which answers the first part of \cite[Conjecture 1.2]{GN2} replacing their $\Hcal_{m,n}$ by the derived version of their $\Hcal'_{m,n}$  {(these sheaves are conjecturally the same up to a grading shift, cf. \cite[Conjecture 2.2]{GN2}). Moreover, by \cite[Proposition 3.8]{ma}, $\mathcal{H}'_{m,n}$ is the Hodge associated graded of $\overline{\mathbb{N}}_m$ if $\operatorname{GCD}(m,n)=1$}.

The description of the decomposition can be made more explicit in the style of
\cite{HK}.  
Let $\Ffrak$ denote the Fourier transform on $\Dcal_\gfrak$-modules. For any $\tau=(\tau_i)\vdash d$, let $O_{r\tau}\subset \gfrak$ denote the nilpotent orbit corresponding to { $r\tau:=(r\tau_1,\ldots, r\tau_k)$}. Let $b=\gcd(\tau_i)$. Then $\pi_1(O_{r\tau})=\pi_0(\Stab_{\SL_n}(x))=\Z/br\Z$ for an element $x\in O_{r\tau}$. For an integer $a$ coprime to $r$, let $\Kbb_{a,\tau}$ be the integrable connection on $O_{r\tau}$ corresponding to $e^{2\pi a\sqrt{-1}/r}\in (\Z/br\Z)^*$. Let $\tau^t$ denote the transpose of $\tau$.
\begin{prop}
$\Ffrak(IC(\gfrak,\Ebb_{a,\tau}))=IC(\gfrak,\Kbb_{a,\tau^t})$.
\end{prop}
\begin{proof}
First consider the case when $d=1$. Here let $\Nbb_{a}$ denote the $\Dcal_{\sln}$-module analogous to the group version in Section \ref{SS:cuspidal}. Under the decomposition $\gl_n=\sln\oplus \zfrak$, we have $IC(\gfrak,\Ebb_{a,(1)})=\Nbb_a\boxtimes \Ocal_\zfrak$. Let $z$ be the coordinate of $\zfrak$ and $\delta_\zfrak:=\Dcal_\zfrak/\Dcal_\zfrak z$. Then $IC(\gfrak,\Kbb_{a,(1)})=\Nbb_a\boxtimes \delta_\zfrak$. Let $\Ffrak_{\sln}$ denote the Fourier transform for $\Dcal_{\sln}$-modules. Then the proposition in this case is implied by the isomorphism $\Ffrak_{\sln}(\Nbb_{a})\cong \Nbb_{a}$. This is shown in the proof of \cite[Proposition 6.4]{BezLosev}.

{  We proceed to the case $d>1$.} Let $\nfrak_\pfrak$ be the nilpotent radical of $\pfrak$ with Levi decomposition $\pfrak=\mathfrak{l}\ltimes \nfrak_\pfrak$. Let $O_{(r)}\subset \gl_r$ be the regular nilpotent orbit. Define $\wt{\gfrak}_d:= G\times^P\left(\tfrak_d + O_{(r)}^d+ \nfrak_\pfrak \right)$ with $\pi_1(\wt{\gfrak}_d)=\Z/r\Z$. Let $\wt{\Ebb}_a$ be the integrable connection on $\wt{\gfrak}_d$ corresponding to $e^{2\pi a\sqrt{-1}/r}$ in $(\Z/r\Z)^*$. Let $p:G\times^P \pfrak\to \gfrak$ be the action map and $j: \wt{\gfrak}_d\to G\times^P \pfrak$ be the embedding. By definition, $\ol{\Nbb}_{a,(r)}^{* d}=p_\dagger j_\dagger \wt{\Ebb}_a$.

Let $\wt{\Kbb}_a$ be the restriction of $\wt{\Ebb}_a$ to $G\times^P\left( O_{(r)}^d+ \nfrak_\pfrak\right)$. Let $\Ffrak_{{\gfrak}}$ denote the partial Fourier transform on $G\times^P {\gfrak}$ along the ${\gfrak}$ direction. Note that $(\gfrak/\pfrak)^*\cong \nfrak_\pfrak$ under the trace form. This plus the proposition in the case of $d=1$ implies that $\Ffrak_\gfrak(IC(G\times^P \gfrak, \wt{\Ebb}_a))=IC(G\times^P \gfrak, \wt{\Kbb}_a)$ (the same holds if we replace the intermediate extensions with the $\dagger$- or $!$-extensions in both parts). By Springer theory, there is an $S_d$ action on $q_\dagger IC(G\times^P \gfrak, \wt{\Kbb}_a)$ for $q: G\times^P \gfrak\to \gfrak$. Similar to \cite[Theorem 5.2]{HK}, there is an isomorphism of $S_d$-modules:
	\[
    \Ffrak q_\dagger(IC(G\times^P \gfrak, \wt{\Ebb}_a))=q_\dagger IC(G\times^P \gfrak, \wt{\Kbb}_a)\otimes \sign.
\]
As $V_\tau\otimes \sign\cong V_{\tau^t}$, it remains to show the $S_d$-module decomposition
	\[
    q_\dagger IC(G\times^P \gfrak, \wt{\Kbb}_a)\cong \bigoplus_{\tau\vdash d} V_\tau\otimes IC(\gfrak,\Kbb_{a,\tau}).
\]
For this, we apply the generalized Springer correspondence \cite[Proposition 5.2]{Lusztig-Spaltenstein} for the cuspidal datum
	\[
    \bigl(L\cap\SL_n,\, O_{(r)}^d,\, \Kbb_a^{\boxtimes d}\bigr),
\]
which has relative Weyl group $S_d$. It identifies the multiplicity space attached to $(O_{r\tau},\Kbb_{a,\tau})$ with the irreducible representation $V_\tau$.
\end{proof}

\end{document}